\documentclass{article}

\usepackage{amsmath,amsthm,amssymb,tikz,bbm,float}
\usepackage[margin=1in]{geometry}

\usepackage{appendix}

\newtheorem{theorem}{Theorem}
\numberwithin{theorem}{section}
\newtheorem{lemma}{Lemma}
\numberwithin{lemma}{subsubsection}
\newtheorem{proposition}{Proposition}
\numberwithin{proposition}{subsubsection}
\newtheorem{corollary}{Corollary}
\numberwithin{corollary}{subsubsection}
\theoremstyle{definition}

\newtheorem{example}{Example}
\newtheorem{note}{Note}
\newtheorem*{question}{Question}

\title{Modular fusion categories with few twists}
\author{Andrew Schopieray}
\date{}

\begin{document}

\maketitle

\begin{abstract}
We classify modular fusion categories up to braided equivalence with less than four distinct twists of simple objects by observing that under this assumption, for each positive integer $N$, there are finitely many modular fusion categories of Frobenius-Schur exponent $N$ up to braided equivalence whose twists are a proper subset of the $N^\mathrm{th}$ roots of unity.
\end{abstract}

\section{Introduction}

A modular fusion (tensor) category \cite[Definition 8.13.4]{tcat} determines a $(2+1)$-dimensional topological quantum field theory, i.e.\ a monoidal functor from the category
of 3-dimensional cobordisms to the category of finite dimensional vector spaces.  This  naturally provides a projective representation of the mapping class group of closed surfaces.  The modular data of the modular fusion category is the data of this projective representation associated to the torus, whose mapping class group is the modular group $\mathrm{SL}(2,\mathbb{Z})$; this data is central to the role of modular fusion categories in mathematical physics and
low-dimensional topology.  Though not determining the modular fusion category uniquely \cite{data}, the modular data determines many characteristics of the modular fusion category such as rank, global dimension, fusion rules, etc.\ and is often reported as a pair of matrices of cyclotomic integers $S,T$ with $T$ a diagonal matrix of roots of unity $\theta_X$ for all simple objects $X$: the \emph{twists} of the modular fusion category.  It was shown that once $S,T$ are rescaled appropriately to a normalized pair $s,t$, the corresponding $\mathrm{SL}(2,\mathbb{Z})$-representation can be factored through a congruence subgroup representation.  That is to say one has an $\mathrm{SL}(2,\mathbb{Z}/n\mathbb{Z})$-representation where $n\in\mathbb{Z}_{\geq1}$ is the order of the matrix $t$.  Irreducible $\mathrm{SL}(2,\mathbb{Z}/n\mathbb{Z})$-representations for any $n\in\mathbb{Z}_{\geq1}$ have been completely described \cite{MR444787,MR444788}, so computational methods can be utilized towards classification efforts by rank, i.e.\ the number of isomorphism classes of simple objects, leveraging the copious restrictions on the representation coming from the underlying categorical structure.  For a modular fusion category $\mathcal{C}$, we will universally use $N\in\mathbb{Z}_{\geq1}$ to denote the order of the $T$-matrix of $\mathcal{C}$ which is also known as the Frobenius-Schur exponent $\mathrm{FSexp}(\mathcal{C})$ \cite{MR2313527}.  Classifying modular fusion categories by $N$ is more complex than a study by rank in the sense that there exist infinitely many inequivalent modular fusion categories $\mathcal{C}$ with any given Frobenius-Schur exponent greater than $1$, e.g.\ untwisted doubles of finite groups of exponent $N$.  Nonetheless, we show there are finitely many modular fusion categories with less than four distinct twists modulo the infinite familes of modular fusion subcategories of twisted doubles of finite groups of exponent $2$ and $3$.  These finitely many sporadic categories are explicitly listed up to braided equivalence in Figures \ref{fig:two} and \ref{fig:three}.

\par One striking observation from this novel classification is that the sporadic, or exceptional, categories with two and three twists (those with $N\neq2,3$, respectively) can be characterized in an unexpected way.  These are precisely the examples for which the sets of twists are \emph{incomplete}.  In other words, when restricted to less than four twists, for each $N\in\mathbb{Z}_{\geq1}$, there are finitely many modular fusion categories $\mathcal{C}$ up to braided equivalence with $\mathrm{FSexp}(\mathcal{C})=N$ whose twists are a proper subset of the $N^\mathrm{th}$ roots of unity.  To illustrate how rare modular fusion categories with incomplete twists are in this setting, consider all modular fusion categories $\mathcal{C}$ with $\mathrm{FSexp}(\mathcal{C})=N=4$.  Such a category is nilpotent \cite{MR2383894}, and therefore must be a fusion subcategory of a twisted double $\mathcal{Z}(\mathrm{Vec}_G^\omega)$ of a finite group $G$ of exponent $2$ or $4$ for some $\omega\in H^3(G,\mathbb{C}^\times)$ \cite[Theorem 1.2]{drinfeld2007grouptheoretical}.  Despite the inordinate number and variety of such modular fusion categories, our results prove there are exactly seven up to braided equivalence whose twists form an incomplete set: the two metric groups of rank $2$ (Theorem \ref{th:twotwists}), four metric groups of ranks $4$ and $16$ (Theorem \ref{th:thre} and Example \ref{groupexamples}), and exactly one twisted double of the elementary abelian $2$-group $C_2^3$.  The latter is the largest of the incomplete examples with less than four twists.   It has rank $22$ and dimension $64$, and its modular data is given in Figure \ref{fig:twoo} below.

\begin{figure}[H]
\centering
\begin{equation*}\arraycolsep=1pt\def\arraystretch{1}
S=\left[\begin{array}{rrrrrrrrrrrrrrrrrrrrrr}
1 & 1 & 1 & 1 & 1 & 1 & 1 & 1 & 2 & 2 & 2 & 2 & 2 & 2 & 2 & 2 & 2 & 2 & 2 & 2 & 2 & 2 \\
1 & 1 & 1 & 1 & 1 & 1 & 1 & 1 & -2 & -2 & 2 & 2 & 2 & 2 & -2 & -2 & -2 & -2 & 2 & 2 & -2 & -2 \\
1 & 1 & 1 & 1 & 1 & 1 & 1 & 1 & 2 & 2 & -2 & -2 & 2 & 2 & -2 & -2 & 2 & 2 & -2 & -2 & -2 & -2 \\
1 & 1 & 1 & 1 & 1 & 1 & 1 & 1 & -2 & -2 & -2 & -2 & 2 & 2 & 2 & 2 & -2 & -2 & -2 & -2 & 2 & 2 \\
1 & 1 & 1 & 1 & 1 & 1 & 1 & 1 & 2 & 2 & 2 & 2 & -2 & -2 & 2 & 2 & -2 & -2 & -2 & -2 & -2 & -2 \\
1 & 1 & 1 & 1 & 1 & 1 & 1 & 1 & -2 & -2 & 2 & 2 & -2 & -2 & -2 & -2 & 2 & 2 & -2 & -2 & 2 & 2 \\
1 & 1 & 1 & 1 & 1 & 1 & 1 & 1 & 2 & 2 & -2 & -2 & -2 & -2 & -2 & -2 & -2 & -2 & 2 & 2 & 2 & 2 \\
1 & 1 & 1 & 1 & 1 & 1 & 1 & 1 & -2 & -2 & -2 & -2 & -2 & -2 & 2 & 2 & 2 & 2 & 2 & 2 & -2 & -2 \\
2 & -2 & 2 & -2 & 2 & -2 & 2 & -2 & -4 & 4 & 0 & 0 & 0 & 0 & 0 & 0 & 0 & 0 & 0 & 0 & 0 & 0 \\
2 & -2 & 2 & -2 & 2 & -2 & 2 & -2 & 4 & -4 & 0 & 0 & 0 & 0 & 0 & 0 & 0 & 0 & 0 & 0 & 0 & 0 \\
2 & 2 & -2 & -2 & 2 & 2 & -2 & -2 & 0 & 0 & -4 & 4 & 0 & 0 & 0 & 0 & 0 & 0 & 0 & 0 & 0 & 0 \\
2 & 2 & -2 & -2 & 2 & 2 & -2 & -2 & 0 & 0 & 4 & -4 & 0 & 0 & 0 & 0 & 0 & 0 & 0 & 0 & 0 & 0 \\
2 & 2 & 2 & 2 & -2 & -2 & -2 & -2 & 0 & 0 & 0 & 0 & -4 & 4 & 0 & 0 & 0 & 0 & 0 & 0 & 0 & 0 \\
2 & 2 & 2 & 2 & -2 & -2 & -2 & -2 & 0 & 0 & 0 & 0 & 4 & -4 & 0 & 0 & 0 & 0 & 0 & 0 & 0 & 0 \\
2 & -2 & -2 & 2 & 2 & -2 & -2 & 2 & 0 & 0 & 0 & 0 & 0 & 0 & -4 & 4 & 0 & 0 & 0 & 0 & 0 & 0 \\
2 & -2 & -2 & 2 & 2 & -2 & -2 & 2 & 0 & 0 & 0 & 0 & 0 & 0 & 4 & -4 & 0 & 0 & 0 & 0 & 0 & 0 \\
2 & -2 & 2 & -2 & -2 & 2 & -2 & 2 & 0 & 0 & 0 & 0 & 0 & 0 & 0 & 0 & -4 & 4 & 0 & 0 & 0 & 0 \\
2 & -2 & 2 & -2 & -2 & 2 & -2 & 2 & 0 & 0 & 0 & 0 & 0 & 0 & 0 & 0 & 4 & -4 & 0 & 0 & 0 & 0 \\
2 & 2 & -2 & -2 & -2 & -2 & 2 & 2 & 0 & 0 & 0 & 0 & 0 & 0 & 0 & 0 & 0 & 0 & -4 & 4 & 0 & 0 \\
2 & 2 & -2 & -2 & -2 & -2 & 2 & 2 & 0 & 0 & 0 & 0 & 0 & 0 & 0 & 0 & 0 & 0 & 4 & -4 & 0 & 0 \\
2 & -2 & -2 & 2 & -2 & 2 & 2 & -2 & 0 & 0 & 0 & 0 & 0 & 0 & 0 & 0 & 0 & 0 & 0 & 0 & -4 & 4 \\
2 & -2 & -2 & 2 & -2 & 2 & 2 & -2 & 0 & 0 & 0 & 0 & 0 & 0 & 0 & 0 & 0 & 0 & 0 & 0 & 4 & -4
\end{array}\right]
\end{equation*}
\begin{equation*}
T=\mathrm{diagonal}(1, 1, 1, 1, 1, 1, 1, 1, \zeta_4^3, \zeta_4, \zeta_4^3, \zeta_4, \zeta_4^3, \zeta_4, \zeta_4^3, \zeta_4, \zeta_4^3, \zeta_4, \zeta_4^3, \zeta_4, \zeta_4^3, \zeta_4 ),\text{ with }\zeta_4:=\exp(\pi i/2)
\end{equation*}
    \caption{The modular data of the largest incomplete modular fusion category with less than $4$ twists}%
    \label{fig:twoo}%
\end{figure}

\par One may be tempted from this data to conjecture that only finitely many modular fusion categories exist up to braided equivalence with a fixed Frobenius-Schur exponent and incomplete sets of twists, but even for four distinct twists this is false by taking products of modular fusion categories with complete sets of twists with those with incomplete sets.  By insisting the modular fusion category has no nontrivial factorizations, i.e.\ it is prime, the following question is natural.

\begin{question}
Does there exist an infinite family of prime modular fusion categories of fixed Frobenius-Schur exponent $N\in\mathbb{Z}_{\geq1}$ whose twists are proper subsets of the $N^\mathrm{th}$ roots of unity?
\end{question}

\par A powerful application of our classification is proof that vast infinite families of representations $\rho_\mathcal{C}'$ of $\mathrm{SL}(2,\mathbb{Z}/n\mathbb{Z})$ for $n\in\mathbb{Z}_{\geq1}$ cannot arise from the modular group representation $\rho_\mathcal{C}$ of a modular fusion category $\mathcal{C}$ (see Section \ref{seccong} for details).  This is a problem that dates back over $30$ years to \cite{MR1227715}, but has made a strong resurgence in the classification of modular fusion categories by rank (e.g.\  \cite{MR4630478,ng2023classificationmodulardatarank}) and other related studies (e.g.\ \cite{MR4389082,MR4834521,MR4793461}).  One of the technical challenges is the appearance of $t$-eigenvalues with high multiplicity.  The reader may reference \cite[Section V.A]{ng2023classificationmodulardatarank} to see these challenges in practice.  Our approach faces these difficulties head-on and essentially demands that $t$-eigenvalues appear with high multiplicity.  For example, there exist irreducible $\mathrm{SL}(2,\mathbb{Z}/4\mathbb{Z})$-representations with $t$-eigenvalues $\{\zeta_4\}$, $\{\zeta_4,\zeta_4^3\}$, and $\{1,\zeta_4,\zeta_4^3\}$; any nonnegative integer linear combination of these can potentially be realized as $\rho_\mathcal{C}'$ for a modular fusion category $\mathcal{C}$ with either two or three twists.  There are no results in the current literature showing that the number of realizable linear combinations is even finite.  But our classification dictates prescisely which linear combinations are realizable.  The interested reader could take any nonnegative integer linear combination of the irreducible $\mathrm{SL}(2,\mathbb{Z}/n\mathbb{Z})$-representions from Lemmas \ref{one} and \ref{two} with combined at most three distinct $t$-eigenvalues and reach similar conclusions.

\par The manuscript contains a brief recollection and thorough list of sources of the definitions and tools needed from the study of modular fusion categories in Sections \ref{sec:mood}--\ref{seccong} (and Appendix \ref{appa}), followed by an alternative proof of the classification of spherical fusion categories of Frobenius-Schur exponent $2$ \cite{MR4202289} in Section \ref{sec:order2}.  The main content of our novel classification is Section \ref{sec:tootwist} where one can find a description of the exceptional categories with incomplete sets of twists in Figures \ref{fig:two} and \ref{fig:three}.  The notation used throughout is consistent and nearly standard, albeit terse since some formulas and equations are quite lengthy.  We collect the commonly used symbols in Figure \ref{fig:notation} for continued reference of the reader.

\begin{figure}[H]
\centering
\begin{equation*}
\begin{array}{|c|c|}
\hline \mathrm{Notation} & \mathrm{Meaning} \\\hline\hline
\mathcal{C} &   \text{modular fusion category} \\
\mathcal{O}(\mathcal{C}) & \text{set of isomorphism classes of simple objects of }\mathcal{C} \\
S,T & \text{unnormalized modular data of }\mathcal{C} \\
s,t & \text{normalized modular data of }\mathcal{C}\\
\theta_X,t_X & \text{eigenvalues of }T\text{ and }t\text{ corresponding to simple }X\in\mathcal{C},\text{ respectively}\\
N,n & \text{orders of }T\text{ and }t,\text{ respectively} \\
\rho_\mathcal{C} & \mathrm{SL}(2,\mathbb{Z})-\text{representation defined by }s,t \\
\rho_\mathcal{C}' & \mathrm{SL}(2,\mathbb{Z}/n\mathbb{Z})-\text{representation defined by }s,t \\
\mathcal{C}_\zeta & \text{set of isomorphism classes of simple }X\in\mathcal{C}\text{ with }\theta_X=\zeta \\
D_\zeta & \text{sum of squares of categorical dimensions of }X\in\mathcal{C}_\zeta \\
D & \text{global dimension of }\mathcal{C} \\
\tau_m & m^\mathrm{th}\text{ Gauss sum of }\mathcal{C};\text{ typically }m=\pm1 \\
\xi & \text{first multiplicative central charge of }\mathcal{C} \\
\gamma & \text{chosen cube root of }\xi \\
\mathcal{Z}(\mathcal{C}) & \text{double or ``center'' of }\mathcal{C} \\
\zeta_m & \text{primitive }m^\mathrm{th}\text{ root of unity }\exp(2\pi i/m) \\\hline
\end{array}
\end{equation*}
    \caption{Recurring notation}%
    \label{fig:notation}%
\end{figure}

\section{Modular fusion categories}\label{sec:mood}

\par We will primarily adhere to the definitions and notation found in the standard textbook \cite{tcat} wherein a modular tensor category is a fusion category \cite[Definition 4.1.1]{tcat} with a spherical structure \cite[Definition 4.7.14]{tcat} and nondegenerate braiding \cite[Section 8.20]{tcat}.  In particular, the set of isomorphism classes of simple objects of a modular fusion category $\mathcal{C}$ will be denoted $\mathcal{O}(\mathcal{C})$ for brevity, and is assumed finite.  All categories will be considered over the complex numbers $\mathbb{C}$.  Our primary digression from \cite{tcat} will be to use the term \emph{modular fusion category} in place of modular tensor category, and leave the latter to include the study of similar non-semisimple categories, or those with $|\mathcal{O}(\mathcal{C})|$ infinite.  Essentially none of the categorical minutiae of modular fusion categories will be used explicitly in arguments moving forward so there is no need to expound these details here.  There are five main tools used in the arguments below.  The most important is linear algebra and representation theory: the modular data and congruence subgroup representation(s) corresponding to to a modular fusion category which deserve their own section (Section \ref{seccong}).  Two concepts which are needed to a lesser extent are categorical: results specific to the modular data of doubles of spherical fusion categories (Section \ref{subsecdouble}), and group actions and quotients of modular fusion categories (Section \ref{subsecgroup}).  These sections will introduce notation and vocabulary for the necessary modular fusion categories found in the statements of Theorems \ref{th:twotwists} and \ref{th:thre}, and the curious reader may reference the given sources for technical details.  Lastly, we will require sporadic use of the (categorical) Galois action on $\mathcal{O}(\mathcal{C})$ \cite{gannoncoste} which arises from a (numerical) Galois action on the modular data of $\mathcal{C}$.  An important consequence of this categorical Galois action is that certain numerical invariants of modular fusion categories are algebraic $d$-numbers in the sense of \cite{codegrees}.  We include both of these final topics in an appendix for the uninitiated reader (Appendix \ref{appa}).


\subsection{Modular data of doubles}\label{subsecdouble}

\par To each spherical fusion category $\mathcal{C}$ one can associate a modular fusion category $\mathcal{Z}(\mathcal{C})$, known as the double or center of $\mathcal{C}$ \cite[Sections 7.13 \& 8.5]{tcat} with trivial central charge and global dimension $\dim(\mathcal{Z}(\mathcal{C}))=\dim(\mathcal{C})^2$.  The computation of $\mathcal{Z}(\mathcal{C})$ given $\mathcal{C}$ is not straightforward; even the growth of the rank of $\mathcal{Z}(\mathcal{C})$ relative to that of $\mathcal{C}$ is unknown in generality at this time.  Simple objects of $\mathcal{Z}(\mathcal{C})$ are not-necessarily-simple objects of $\mathcal{C}$ with data describing how this object commutes with others in $\mathcal{C}$ \cite[Definition 7.13.1]{tcat}.  As such, there is an obvious forgetful tensor functor $F:\mathcal{Z}(\mathcal{C})\to\mathcal{C}$ remembering only the object itself.  The adjoint of this functor $I:\mathcal{C}\to\mathcal{Z}(\mathcal{C})$, called the induction functor \cite[Section 9.2]{tcat}, is a main source of information regarding $\mathcal{Z}(\mathcal{C})$.  For example, each nontrivial $X\in\mathcal{O}(\mathcal{C})$ must satisfy the zero trace condition \cite[Theorem 2.5]{MR3427429}
\begin{equation}\label{zerotrace}
0=\sum_{Y\in\mathcal{O}(\mathcal{Z}(\mathcal{C}))}[Y:I(X)]\dim(Y)\theta_Y,
\end{equation}
while each $X\in\mathcal{O}(\mathcal{C})$ with $X\cong X^\ast$ must satisfy the squared trace condition \cite[Theorem 2.7]{MR3427429}
\begin{equation}\label{zero2trace}
\pm\dim(\mathcal{C})=\sum_{Y\in\mathcal{O}(\mathcal{Z}(\mathcal{C}))}[Y:I(X)]\dim(Y)\theta_Y^2,
\end{equation}
where $[Y:I(X)]=\dim_\mathbb{C}\mathrm{Hom}_{\mathcal{Z}(\mathcal{C})}(Y,I(X))$.  Such results are potent constraints on the types of modular fusion categories which can arise as doubles of spherical fusion categories.

\begin{example}\label{exgroup}
Fusion categories $\mathcal{C}$ such that all $X\in\mathcal{O}(\mathcal{C})$ are invertible, i.e.\ $X\otimes X^\ast\cong\mathbbm{1}$, are known as pointed fusion categories and are monoidally equivalent to the categories $\mathrm{Vec}_G^\omega$ of finite-dimensional $G$-graded vector spaces for a finite group $G$ where the associativity of the tensor product has been twisted by some $\omega\in H^3(G,\mathbb{C}^\times)$.  Each pointed fusion category has a canonical spherical structure \cite[Corollary 9.6.6]{tcat} such that the categorical dimensions of all simple objects are $1$.  On one hand, the doubles $\mathcal{Z}(\mathrm{Vec}_G^\omega)$ of these spherical fusion categories, often called twisted doubles of finite groups, are classical and often overlooked as being fully understood --- indeed, the modular data of such categories can be computed from explicit formulas \cite{MR1770077,MR4257620}.  On the other hand, very basic but important observations about this infinite family of modular fusion categories have only recently come to light.  For example, it has been shown that modular data is not a complete invariant of modular fusion categories using examples of twisted doubles of finite groups \cite[Corollary 4.2]{data}.  The current exposition gives another way in which these categories produce examples exhibiting novel properties.
\end{example}


\subsection{Group actions and de-equivariantization}\label{subsecgroup}

\par The details of this topic can be found in \cite[Section 8.23]{tcat} and Sections 2, 3, and 5 of \cite{MR3039775}.  De-equivariantization can be used as a means of quotienting modular fusion categories by symmetry to acquire a modular fusion category of smaller dimension, making it an indispensable tool akin to the ability to quotient groups by normal subgroups.  This analogy is literal when applied to the study of modular fusion categories associated to finite groups \cite[Example 8.23.7]{tcat}.

\par Given a modular fusion category $\mathcal{C}$ and a connected \'etale algebra $A$ in $\mathcal{C}$ \cite[Section 3.1]{MR3039775}, one can consider the category of local $A$-module objects  in $\mathcal{C}$, denoted $\mathcal{C}_A^0$ \cite[Section 3.5]{MR3039775}.  The carefully chosen hypotheses on the algebra $A$ guarantee that $\mathcal{C}_A^0$ is again a modular fusion category \cite[Corollary 3.30]{MR3039775}, and that \cite[Theorem 4.5]{MR1936496}
\begin{equation}
\dim(\mathcal{C}_A^0)=\dfrac{\dim(\mathcal{C})}{\dim(A)^2}
\end{equation}
This construction gives a well-defined equivalence relation on the collection of all modular fusion categories up to braided equivalence. Specifically, if $\mathcal{C}$ and $\mathcal{D}$ are modular fusion categories, then we say $\mathcal{C}$ and $\mathcal{D}$ are \emph{Witt equivalent} if there exist connected \'etale algebras $A$ and $B$ in $\mathcal{C}$ and $\mathcal{D}$, respectively, such that $\mathcal{C}_A^0\simeq\mathcal{D}_B^0$ is a braided equivalence \cite[Definition 5.1 \& Proposition 5.15]{MR3039775}.  Each Witt equivalence class contains a unique representative up to braided equivalence which cannot be reduced further \cite[Theorem 5.13]{MR3039775}, i.e.\ it contains no nontrivial connected \'etale algebras.

\begin{example}
A pointed (see Example \ref{exgroup}) modular fusion category is known as a metric group and is defined by a finite group $G$ (necessarily abelian) and a nondegenerate quadratic form $q:G\to\mathbb{C}^\times$ \cite[Section 8.4]{tcat}.  For this reason we will use the notation $\mathcal{C}(G,q)$ to refer to such modular fusion categories.  The Witt equivalence classes of metric groups are completely known \cite[Section 5.3]{MR3039775} and can be described by the classes for $p$-groups for primes $p$; for $p=2$ there are $16$ distinct classes while for odd $p$ there are $4$ distinct classes.
\end{example}

\begin{example}
Let $\mathcal{C}$ be a modular fusion category which is \emph{weakly group-theoretical} in the sense of \cite[Definition 9.8.1]{tcat}.  These are roughly the modular fusion categories which can be constructed using the data of finite groups.  Weakly group-theoretical categories include those which are solvable \cite{solvable} and nilpotent \cite{MR2383894} and this property is preserved under the de-equivariantization process outlined above.  It was shown in \cite[Theorem 1.1]{MR3770935} that, up to braided equivalence, the unique minimal representative of the Witt class of a weakly group-theoretical modular fusion category is the product of a metric group and one of the eight braided equivalence classes of Ising fusion categories \cite[Corollary B.16]{DGNO}.
\end{example}


\section{Congruence subgroup representations}\label{seccong}

Let $\mathcal{C}$ be a modular fusion category and $\rho_\mathcal{C}$ be its associated $\mathrm{SL}(2,\mathbb{Z})$-representation defined by a normalized modular data $s,t$ as in \cite[Remark 8.16.2]{tcat}.  In particular we will denote the first multiplicative central charge of $\mathcal{C}$ as $\xi:=\xi(\mathcal{C})$ and the chosen cube root of $\xi$ by $\gamma$.  It is known that $\rho_\mathcal{C}$ can be factored through $\mathrm{SL}(2,\mathbb{Z}/n\mathbb{Z})$ where again $n\in\mathbb{Z}_{\geq1}$ is the order of $t$ \cite[Theorem 6.8]{MR2725181}.  Thus to each modular fusion category $\mathcal{C}$ one can associate a representation $\rho'_\mathcal{C}$ of the finite group $\mathrm{SL}(2,\mathbb{Z}/n\mathbb{Z})$ such that $\rho_\mathcal{C}=\rho'_\mathcal{C}\pi$ where $\pi:\mathrm{SL}(2,\mathbb{Z})\to\mathrm{SL}(2,\mathbb{Z}/n\mathbb{Z})$ is the canonical projection.  The benefit being that the irreducible finite-dimensional complex representations of $\mathrm{SL}(2,\mathbb{Z}/n\mathbb{Z})$ have been described explicitly, up to isomorphism \cite{MR444787,MR444788}.  These irreducible representations of $\mathrm{SL}(2,\mathbb{Z}/n\mathbb{Z})$ can be described from the irreducible representations of $\mathrm{SL}(2,\mathbb{Z}/p^\lambda\mathbb{Z})$ for $p^{\lambda}$ dividing $n$ where $p\in\mathbb{Z}_{\geq2}$ is prime and $\lambda\in\mathbb{Z}_{\geq1}$ (refer to \cite[Section 3]{MR1354262}).  We will often refer to the $t$-eigenvalues or $t$-spectrum of a modular fusion category, or of the representations $\rho_\mathcal{C}$ or the congruence subgroup representation $\rho_\mathcal{C}'$.  In the case $\rho_\mathcal{C}'$ is reducible, we may refer to the $t$-spectrum of individual summands which should not cause any confusion in context.
\par The following (Lemmas \ref{one} and \ref{two}) is a complete list of isomorphism classes of irreducible $\mathrm{SL}(2,\mathbb{Z}/n\mathbb{Z})$-representations with less than four $t$-eigenvalues.  Most notably, every irreducible representation listed below is nondegenerate in the sense that its $t$-spectrum is multiplicity-free.  This is a crucial observation for future use of results such as \cite[Lemma 5.2.2]{MR4834521}.

\begin{lemma}\label{one}
If $\rho$ is an irreducible $\mathrm{SL}(2,\mathbb{Z}/n\mathbb{Z})$-representation for some $n\in\mathbb{Z}_{\geq1}$ with two or fewer $t$-eigenvalues, then $\rho$ is isomorphic to one of the representations listed in Figure \ref{fig:twots} or a tensor product of these with a one-dimensional representation of coprime level.
\begin{figure}[H]
\centering
\begin{equation*}
\begin{array}{|cccc|}\hline
 \mathrm{Level} & \mathrm{Name} & \#/\cong &t-\mathrm{spectra} \\\hline
12/\gcd(12,j),\,\,0\leq j\leq11 & C_j & 12 & \{\zeta_{12}^j\}\\\hline
2 & N_1(\chi_1) & 1 & \{-1,1\} \\
3 & N_1(\chi) & 1 & \{\zeta_3,\zeta_3^2\}\\
3 & N_1(\chi)\otimes C_4,N_1(\chi)\otimes C_8 & 2 & \{1,\zeta_3^2\},\{1,\zeta_3\}\\
4 & N_2(\chi) & 1 & \{\zeta_4,\zeta_4^3\} \\
5 & R_1(1,\chi_{-1}),R_2(r,\chi_{-1}) & 2 &\{\zeta_5,\zeta_5^4\},\{\zeta_5^2,\zeta_5^3\} \\
8 & N_3(\chi)_+\otimes C_j & 4 & \{\zeta_8^3,\zeta_8^5\},\{\zeta_8^1,\zeta_8^7\}\\
 &  &  & \{\zeta_8^1,\zeta_8^3\},\{\zeta_8^5,\zeta_8^7\} \\
\hline
\end{array}
\end{equation*}
    \caption{Irreducible\ $\mathrm{SL}(2,\mathbb{Z}/n\mathbb{Z})$-reps.\ with less than three distinct $t$-eigenvalues up to isomorphism.  Their given names are from \cite{MR444787,MR444788} and their levels are the collective order of their $t$-spectra.}%
    \label{fig:twots}%
\end{figure}
\end{lemma}

\begin{proof}
Let $J$ be the set of the prime integers and assume $n=\prod_{p\in J}p^{a_p}$ for some $a_p\in\mathbb{Z}_{\geq0}$.  Each irreducible $\mathrm{SL}(2,\mathbb{Z}/n\mathbb{Z})$-representation factors as $\rho\cong\bigotimes_{p\in J}\rho_{p^{a_p}}$ where $\rho_{p^{a_p}}$ is an irreducible $\mathrm{SL}(2,\mathbb{Z}/p^{a_p}\mathbb{Z})$-representation \cite[Lemma 7]{MR1354262}.  We say that irreducible $\mathrm{SL}(2,\mathbb{Z}/p^{a_p}\mathbb{Z})$-representations are of \emph{level} $p^{a_p}$.  Let $\zeta\in\mathbb{Q}(\zeta_n)$ be a $t$-eigenvalue of $\rho$ where $\zeta_n:=\exp(2\pi i/n)$.

\par If $\zeta$ is the unique $t$-eigenvalue of $\rho$, then each $\rho_{p^{a_p}}$ has a unique $t$-eigenvalue.  By the classification of irreducible representations of level $p^{a_p}$ for all $p\in J$, one finds only one-dimensional representations satisfy this condition.  Moreover $\rho$ is one-dimensional and there are $12$ such isomorphism classes determined by $\zeta$.

\par Now assume $\zeta\neq\zeta'$ are the $t$-eigenvalues of $\rho$.  If all irreducible factors $\rho_{p^{a_p}}$ have a unique $t$-eigenvalue, we have already explained that $\rho$ is one-dimensional, against our assumptions.  On the other hand, if there exist factors $\rho_{p^{a_p}}$ and $\rho_{q^{a_q}}$ for some primes $p\neq q$ both having two or more distinct $t$-eigenvalues, then $\rho$ would have at least four distinct $t$-eigenvalues.  Therefore $\rho$ has a unique factor $\rho_{p^{a_p}}$ possessing exactly two distinct $t$-eigenvalues.  By the classification of irreducible representations of level $p^{a_p}$ for all $p\in J$, one finds only two-dimensional representations satisfy this condition.  Therefore $\rho$ is two-dimensional, and the isomorphism classes of the unique two-dimensional factor can be listed in Figure \ref{fig:twots}.
\end{proof}

\begin{lemma}\label{two}
If $\rho$ is an irreducible $\mathrm{SL}(2,\mathbb{Z}/n\mathbb{Z})$-representation for some $n\in\mathbb{Z}_{\geq1}$ with exactly three distinct $t$-eigenvalues, then $\rho$ is isomorphic to one of the representations in Figure \ref{fig:threets} or a tensor product of these with a one-dimensional representation of coprime level.
\begin{figure}[H]
\centering
\begin{equation*}
\begin{array}{|cccc|}
\hline\mathrm{Level} & \mathrm{Name} & \#/\cong &t-\mathrm{spectra} \\\hline
3 & N_1(\chi_1) & 1 & \{1,\zeta_3,\zeta_3^2\} \\
4 & D_2(\chi)_+,D_2(\chi)_+\otimes C_6 & 2 & \{1,-1,\zeta_4\},\{1,-1,\zeta_4^3\}\\
 &  D_2(\chi)_+\otimes C_9,D_2(\chi)_+\otimes C_3 & 2 & \{1,\zeta_4,\zeta_4^3\},\{-1,\zeta_4,\zeta_4^3\} \\
5 & R_1(1,\chi_1),R_1(2,\chi_1) & 2 & \{1,\zeta_5,\zeta_5^4\},\{1,\zeta_5^2,\zeta_5^3\} \\
7 & R_1(1,\chi_{-1}),R_1(2,\chi_{-1}) & 2 & \{\zeta_7,\zeta_7^2,\zeta_7^4\},\{\zeta_7^3,\zeta_7^5,\zeta_7^6\} \\
8 & R_3^0(1,3,\chi)_\pm,R_3^0(1,3,\chi)_\pm\otimes C_6 & 4 & \{\pm1,\zeta_8,\zeta_8^5\},\{\pm1,\zeta_8^3,\zeta_8^7\} \\
 & R_3^0(1,3,\chi)_\pm\otimes C_3,R_3^0(1,3,\chi)_\pm\otimes C_9 & 4 & \{\pm\zeta_4^3,\zeta_8^3,\zeta_8^7\},\{\mp\zeta_4,\zeta_8,\zeta_8^5\} \\
16 &R_4^0(1,1,\chi)_\pm\otimes C_j,R_4^0(3,1,\chi)_\pm\otimes C_j & 16 & \textnormal{(refer to \cite[Table A.1]{MR3632091})}\\
&\textnormal{for }j=0,3,6,9 & & \\\hline
\end{array}
\end{equation*}
    \caption{Irreducible $\mathrm{SL}(2,\mathbb{Z}/n\mathbb{Z})$-reps.\ with exactly three distinct $t$-eigenvalues up to isomorphism.  Their given names are from \cite{MR444787,MR444788} and their levels are the collective order of their $t$-spectra.}%
    \label{fig:threets}%
\end{figure}
\end{lemma}

\begin{proof}
The proof is the same as that of Lemma \ref{one} since any irreducible representation can only be a nontrivial tensor product of irreducible representations of prime power level with one, two, or three distinct $t$-eigenvalues.
\end{proof}

\begin{proposition}\label{copropprime}
Let $\mathcal{C}$ be a modular fusion category.  If all $t$-eigenvalues of $\mathcal{C}$ have coprime order, then $\theta_X\in\{1,-1\}$ for all simple $X\in\mathcal{C}$.
\end{proposition}

\begin{proof}
Recall the Galois action on $t$-eigenvalues in a modular fusion category: for each $\sigma\in\mathrm{Gal}(\overline{\mathbb{Q}}/\mathbb{Q})$ and $X\in\mathcal{O}(\mathcal{C})$, $t_{\hat{\sigma}(X)}=\sigma^2(t_X)$ \cite[Theorem II(iii)]{dong2015congruence} (see also Appendix \ref{appa}).  The coprime assumption implies the $t$-eigenvalues of $\mathcal{C}$ are individually fixed under square Galois conjugacy, i.e.\ we must have $t_X^{24}=1$ for all simple $X$ in $\mathcal{C}$.  Moreover there are at most three distinct $t$-eigenvalues of $\mathcal{C}$, and there is a very small list of possible irreducible summands of $\rho_\mathcal{C}'$ found in Lemmas \ref{one} and \ref{two}.  Any tensor product of coprime irreducible summands of $\rho_\mathcal{C}'$ would produce $t$-eigenvalues which are not coprime.  Hence any irreducible summand of $\rho_\mathcal{C}'$ of dimension greater than $1$ is either the unique two-dimensional level $2$ representation with $t$-spectrum $\{1,-1\}$ or the two-dimensional level $3$ representations $\{1,\omega\}$ for a primitive third root of unity $\omega$.  If $\rho_\mathcal{C}'$ is a sum of one-dimensional representations, then $\rho_\mathcal{C}'\cong m\rho$ for a unique one-dimensional representation $\rho$ and $m\in\mathbb{Z}_{\geq1}$ or else \cite[Lemma 3.18]{MR3632091} is violated.  It follows that $\mathcal{C}$ is trivial by \cite[Lemma 5.2.2]{MR4834521}.  If there exists a two-dimensional level $2$ irreducible summand of $\rho_\mathcal{C}'$ but none of level $3$, then any one-dimensional summand of $\rho_\mathcal{C}'$ must have $t$-eigenvalue $\pm1$ or else \cite[Lemma 3.18]{MR3632091} is violated.  We are then done since $N$ divides $n$ \cite[Theorem 7.1]{MR2725181}.  It remains to consider when $\rho_\mathcal{C}'$ has a two-dimensional level $3$ irreducible summand.  In this case, $\mathcal{C}$ is integral by \cite[Theorem 7.1]{MR2725181} since the maximal real subfield of $\mathbb{Q}(\zeta_{12})$ is $\mathbb{Q}$, therefore $\xi^8=1$ \cite[Proposition 2.6]{MR4836055} and moreover $(\gamma^{-1})^8=1$.  The $t$-eigenvalues of $\mathcal{C}$ are a subset of $\{1,-1,\omega\}$, thus $\gamma=\pm1$ and therefore the twists of $\mathcal{C}$ are either $\{1\}$, $\{1,-1\}$, $\{1,-1,\omega\}$, or $\{1,-1,-\omega\}$.  The latter two contradict $\xi=\pm1$ as the first Gauss sum $\tau_1=a+b\omega$ is not real when $b\neq0$.
\end{proof}

The following section describes modular fusion categories satisfying the hypotheses of Proposition \ref{copropprime}.


\section{$N=2$}\label{sec:order2}

Assume $\mathcal{C}$ is a modular fusion category with $N=1$.  Then abbreviating $D:=\dim(\mathcal{C})$ from this point forward, $D=\tau_1\tau_{-1}=D^2$ \cite[Proposition 8.15.4]{tcat}.  Therefore $D=1$ and thus $\mathcal{C}\simeq\mathrm{Vec}$ \cite[Theorem 2.3]{ENO}.  We consider in this section the case when $N=2$.  The results of this section are contained in \cite{MR4202289}, albeit in a lengthier exposition.

\begin{theorem}\label{prop2}
Let $\mathcal{C}$ be a modular fusion category with $N=2$.  Then $D=2^{2m}$ for some $m\in\mathbb{Z}_{\geq1}$, $\xi\in\{-1,1\}$, and $\mathcal{C}$ is unique up to braided equivalence for each chosen pair of $m\in\mathbb{Z}_{\geq1}$ and $\xi\in\{-1,1\}$.
\end{theorem}

\begin{proof}
Assume $\mathcal{C}$ is a modular fusion category such that $\theta_X\in\{-1,1\}$ for all $X\in\mathcal{O}(\mathcal{C})$.  The $S$-matrix entries, and in particular the dimensions of simple objects, are contained in $\mathbb{Q}(\pm1)=\mathbb{Q}$ \cite[Theorem 2.7]{MR3486174}.  Moreover, it is evident that $\tau_1=\tau_{-1}$ hence $D=\tau_1\tau_{-1}$ is a perfect square integer and we conclude $D=2^{2m}$ for some $m\in\mathbb{Z}_{\geq1}$ by \cite[Theorem 3.9]{MR3486174}.  Now since $\mathcal{C}$ is nilpotent \cite[Theorem 8.28]{ENO} and $\xi\in\{-1,1\}$, then when $\xi=1$, $\mathcal{C}\simeq\mathcal{Z}(\mathrm{Vec}_G^\omega)$ is a braided equivalence for a finite $2$-group $G$ and $\omega\in H^3(G,\mathbb{C}^\times)$ \cite[Theorem 1.3]{drinfeld2007grouptheoretical}.  Furthermore, $N=2$ only if $G$ has exponent $2$ \cite[Theorem 9.2]{MR2313527}, i.e.\ is an elementary abelian $2$-group.  Otherwise $\xi=-1$, and $\mathcal{C}\simeq\mathcal{V}\boxtimes\mathcal{Z}(\mathrm{Vec}_H^\omega)$ for $H\cong C_2^{\ell-1}$, $\omega\in H^3(H,\mathbb{C}^\times)$, and $\ell\geq1$ where $\mathcal{V}$ is the rank 4 metric group with $\theta_X=-1$ for all nontrivial invertible $X$ and $C_2^0$ is the trivial group.  Note that $N=2$ for $\mathcal{Z}(\mathrm{Vec}_{C_2^m}^\omega)$ for any $m\in\mathbb{Z}_{\geq1}$ if and only if $\omega$ is trivial \cite[Theorem 4.7]{MR2333187} and so our proof is complete.
\end{proof}

\begin{note}\label{yilongnote}
It was observed in \cite[Theorem 4.2]{MR4202289} that there are exactly two prime braided equivalence classes of this type: $\mathcal{Z}(\mathrm{Vec}_{C_2})$ of rank $4$, with twists $1,1,1,-1$ and the category $\mathcal{V}$ of rank $4$ with twists $1,-1,-1,-1$.  Both categories are pointed and will be referred to in the future with the standard notation for metric groups $\mathcal{C}(G,q)$ where $G$ is a finite group and $q:G\to\mathbb{C}^\times$ a nondegenerate quadratic form \cite[Section 8.4]{tcat}.
\end{note}

\begin{corollary}
Let $\mathcal{C}$ be a spherical fusion category.  If $\mathrm{FSexp}(\mathcal{C})=2$, then $\mathcal{C}\simeq\mathrm{Vec}_G$ is a tensor equivalence for an elementary abelian $2$-group $G$.
\end{corollary}

\begin{proof}
As $\mathrm{FSexp}(\mathcal{C})=\mathrm{FSexp}(\mathcal{Z}(\mathcal{C}))=N$ \cite[Corollary 7.8]{MR2313527}, then $\mathcal{Z}(\mathcal{C})\simeq\mathcal{Z}(\mathrm{Vec}_G)$ is a braided equivalence for an elementary abelian $2$-group by Theorem \ref{prop2}.  The forgetful functor $F:\mathcal{Z}(\mathcal{C})\to\mathcal{C}$ is monoidal, hence for all simple $X\in\mathcal{Z}(\mathcal{C})$, $F(X)^{\otimes2}=F(X^{\otimes2})=\mathbbm{1}$.  Thus $\mathcal{C}\simeq\mathrm{Vec}_G^\omega$ is a tensor equivalence for an elementary abelian $2$-group and $\omega$ is trivial by the reasoning in the end of the proof of Theorem \ref{prop2}.
\end{proof}


\section{Fixed number of distinct twists}\label{sec:tootwist}

Assume $\mathcal{C}$ is a modular fusion category and all simple objects have the same twist $\theta$.  Then we compute using \cite[Proposition 8.15.4]{tcat},
\begin{equation}\label{eq:key}
D=\tau_1\tau_{-1}=D\theta D\theta^{-1}=D^2.
\end{equation}
Therefore $D=1$ and moreover $\mathcal{C}\simeq\mathrm{Vec}$ \cite[Theorem 2.3]{ENO}.  We devote the remainder of this section to classifying modular fusion categories with either $2$ or $3$ distinct twists.  The number-theoretical implications of the relations analogous to Equation (\ref{eq:key}) are the key to obtaining finiteness results with a higher quantity of fixed twists.


\subsection{Two twists}

The entirety of this subsection will prove the following classification.  For brevity, we will use the Lie-theoretic notation for the modular fusion categories of rank $2$.  The interested reader may reference \cite{MR4079742} for a digestible overview of this ubiquitous family of categories.

\begin{theorem}\label{th:twotwists}
Let $\mathcal{C}$ be a modular fusion category with exactly two twists: $1$ and $\theta$.  Then $N=2$ or $\mathcal{C}$ is braided equivalent to one of the categories in Figure \ref{fig:two}.
\end{theorem}
\begin{figure}[H]
\centering
\begin{equation*}
\begin{array}{|c|ccccc|}
\hline\mathcal{C} & \#/\simeq &\textnormal{Indexed by} & \mathrm{FPdim}(\mathcal{C}) & \theta  & N \\\hline\hline
\mathcal{C}(C_2,q) & 2 & q:C_2\to\mathbb{C}^\times & 2 & q(1) & 4 \\
\mathcal{C}(C_3,q) & 2 &q:C_3\to\mathbb{C}^\times &  3 & q(1) & 3 \\
\mathcal{C}(\mathfrak{sl}_2,5,q)_\mathrm{ad} & 4 & \textnormal{primitive } q^{10}=1 & \frac{1}{2}(5+\sqrt{5}) & q^4 & 5\\
\hline
\end{array}
\end{equation*}
    \caption{Modular fusion categories with two distinct twists and $N\neq2$ up to braided equivalence}%
    \label{fig:two}%
\end{figure}

\noindent To justify Theorem \ref{th:twotwists}, define $\mathcal{C}_\zeta:=\{X\in\mathcal{O}(\mathcal{C}):\theta_X=\zeta\}$ and $D_\zeta:=\sum_{X\in\mathcal{C}_\zeta}\dim(X)^2$.  Then one computes from \cite[Proposition 8.15.4]{tcat} and the definition of $\tau_{\pm1}$ \cite[Definition 3.1]{MR3997136},
\begin{align}
D_1+D_\theta=D=\tau_1\tau_{-1}&=(D_1+D_\theta\theta)(D_1+D_\theta\theta^{-1}) =D_1^2+D_\theta^2+D_1D_\theta(\theta+\theta^{-1}).\label{six}
\end{align}
By Galois conjugacy of simple objects (refer to Appendix \ref{appa}), the $t$-eigenvalues of $\mathcal{C}$ are roots of unity of order dividing $2^4\cdot3\cdot5$ by inspecting prime powers $p^k$ with $\varphi(p^k)=(1/2)p^{k-1}(p-1)\leq2$, hence the same is true for $N$ \cite[Theorem 7.1]{MR2725181}.  If $2^4\cdot5$ divides $N$, the $t$-eigenvalues have at least $4$ square Galois conjugates which cannot occur by assumption.  So we divide our argument into the following distinct cases based on the type of argument required: $8$ divides $N$ ($N\in\{8,16,24,40,48,80,120\}$), $5$ divides $N$ and $8$ does not ($N\in\{5,10,15,20,30,60\}$), or neither $5$ nor $8$ divides $N$ ($N\in\{3,4,6,12\}$).  Up to Galois conjugacy of modular fusion categories \cite[Section 4.3]{davidovich2013arithmetic}, we may assume that $\theta$ is the primitive root $\theta=\exp(2\pi i/N):=\zeta_N$, hence $\theta+\theta^{-1}=2\cos(2\pi/N)$.  For each possibility, we have a quadratic relation between $D_1$ and $D_\theta$, and we know both $D_1,D_\theta>0$ by definition.  Specifically, by Equation (\ref{six}),
\begin{align}
&&D_1+D_\theta&=D_1^2+D_\theta^2+D_1D_\theta(\theta+\theta^{-1}) \\
\Rightarrow&&0&=D_1^2+(2\cos(2\pi/N)D_\theta-1)D_1+D_\theta(D_\theta-1)\label{xeq} \\
\Rightarrow&&D_1&=\dfrac{1}{2}(1-2\cos(2\pi/N)D_\theta\pm\sqrt{(2\cos(2\pi/N)D_\theta-1)^2-4D_\theta(D_\theta-1)}).\label{eqx}
\end{align}

\subsubsection{$N\in\{3,4,6,12\}$}

\par If $N=12$, then for $D_1$ to be real in Equation (\ref{eqx}), the discriminant of the quadratic in (\ref{xeq}) must be nonnegative, i.e.\ $-D_\theta^2-2(\sqrt{3}-2)D_\theta+1\geq0$.  This conservatively forces $0<D_\theta<1.304$, and $D_1<0$ in this range, which cannot occur since $D_1\geq1$ \cite[Remark 2.5]{ENO}.  We consider the remaining $N$ using the same reasoning, with the fact that if $N\in\{2,3,4,6\}$, then $\mathcal{C}$ is integral \cite[Theorem 7.1]{MR2725181}.  For $N=3$, $D_\theta\in\{1,2\}$ for the discriminant of the quadratic in (\ref{xeq}) to be nonnegative, and when $N\in\{4,6\}$, $D_\theta=1$.  We compute $D_1=1$ when $N=4$ and $D_\theta=1$.  When $N=6$, there are no positive integer solutions for $D_1$.  Thus the only modular fusion categories with a unique nontrivial twist $\pm\zeta_4$ are the metric groups of rank $2$ by \cite[Example 5.1.2(i)]{ostrikremarks}.  When $N=3$ and $D_\theta=1$, then $D_1=2$.  Any such category is pointed \cite[Example 5.1.2(ii)]{ostrikremarks}, but all such modular fusion categories have $D_1=1$.  When $N=3$ and $D_\theta=2$, then $D_1=2$ or $D_1=1$.  The former is impossible because $3$ is coprime to $D=4$ \cite[Theorem 3.9]{MR3486174}.  The latter has $D=3$ and is therefore a metric group of rank $3$.

\subsubsection{$N\in\{8,16,24,40,48,80,120\}$} \label{sub:3}

\par This is the shortest argument as no such categories exist.  When $N\in\{8,16,24,40,48,80,120\}$ then for $D_1$ to be real, we must have a nonnegative discriminant for the quadratic in Equation (\ref{xeq}), i.e.\
\begin{equation}
(2\cos(2\pi/N)D_\theta-1)^2-4D_\theta(D_\theta-1)\geq0.\label{ten}
\end{equation}
For $N\in\{8,16,24,40,48,80,120\}$ one may verify that $0<D_\theta<7$ in any case, and maximizing $D_1$ over this interval when $D_1\in\mathbb{R}$ for each $N$ shows $D_1<1$.  Moreover no such categories exist by \cite[Remark 2.5]{ENO}.

\subsubsection{$N\in\{5,10,15,20,30,60\}$} 

When $N>5$, an identical argument is valid as in Subsection \ref{sub:3} above.  The case when $N=5$ mimics the argument in Subsection \ref{sub:3}, but a more careful analysis is needed as examples do exist.

\par When $N=5$, then for the inequality in (\ref{ten}) to hold, $0<D_\theta<1.05$ conservatively, and for $D_1>0$, we must have $D_1+D_\theta<2$.  Note that the maximal totally real subfield of $\mathbb{Q}(\zeta_5)$ is $\mathbb{Q}(\sqrt{5})$.  Thus $D$ is an algebraic $d$-number \cite[Corollary 1.4]{codegrees} in $\mathbb{Q}(\sqrt{5})$ which is totally greater than or equal to $1$ \cite[Proposition 7.21.14]{tcat}, but less than $2$.  There are exactly two such numbers: $1$ and $(1/2)(5-\sqrt{5})$ (see Appendix \ref{appa}).  Since $\mathcal{C}$ is not trivial, we conclude that $D=(1/2)(5-\sqrt{5})$ and moreover $\mathcal{C}$ is a Fibonacci modular fusion category \cite[Example 5.1.2(iv)]{ostrikremarks}.  There are four braided equivalence classes of such categories: two equivalence classes of fusion categories each with two distinct braidings \cite{ostrik}.


\subsubsection{Coprime twists}

\par Having completed the proof of Theorem \ref{th:twotwists}, in the final subsection of Section \ref{sec:tootwist}, we note that Theorem \ref{th:twotwists} can be used to generalize Proposition \ref{copropprime} to coprime twists.

\begin{theorem}\label{thmthreetwist}
Let $\mathcal{C}$ be a nontrivial modular fusion category.  If all twists of $\mathcal{C}$ have coprime order, then $\mathcal{C}$ has exactly two distinct twists.  All such modular fusion categories are described in Theorem \ref{th:twotwists}. 
\end{theorem}

\begin{proof}
Let $\gamma$ be a cube root of $\xi$.  Note that $t_{\hat{\sigma}(\mathbbm{1})}=\sigma^2(\gamma^{-1})$ for all $\sigma\in\mathrm{Gal}(\overline{\mathbb{Q}}/\mathbb{Q})$, hence $\theta_{\hat{\sigma}(\mathbbm{1})}=\sigma^2(\gamma^{-1})\gamma$.  If $\gamma^{-1}$ has two or more distinct nontrivial square Galois conjugates, then the corresponding orders of the twists of the simple objects in the Galois orbit of $\mathbbm{1}$ will not be coprime.  In the case $\gamma^{-1}$ has exactly one nontrivial square Galois conjugate, all simple objects of $\mathcal{C}$ must have these two $t$-eigenvalues.  The Fibonacci modular fusion categories $\mathcal{C}(\mathfrak{sl}_2,5,q)_\mathrm{ad}$ are the unique modular fusion categories with this property by Theorem \ref{th:twotwists}.

\par Otherwise, $\sigma^2(\gamma)=\gamma$ for any $\sigma\in\mathrm{Gal}(\overline{\mathbb{Q}}/\mathbb{Q})$, thus
\begin{equation}\theta_{\hat{\sigma}(X)}\gamma^{-1}=t_{\hat{\sigma}(X)}=\sigma^2(t_X)=\sigma^2(\theta_X\gamma^{-1})=\sigma^2(\theta_X)\gamma^{-1}.\end{equation}
Therefore $\theta_{\hat{\sigma}(X)}=\sigma^{2}(\theta_X)$.  Moreover $\theta_X^{24}=1$ for all simple $X\in\mathcal{C}$, and thus by the coprime assumption, as in the proof of Proposition \ref{copropprime}, there are at most three distinct twists including $\theta_\mathbbm{1}=1$.  If there are only two distinct twists, these are described by Theorem \ref{th:twotwists}.  It remains to consider the case there are three distinct twists, say $1$, $\theta_X$ and $\theta_Y$, where $\theta_X^8=1$ and $\theta_Y^3=1$.  In particular, $\xi$ is nontrivial.  We may assume by Galois conjugacy of modular fusion categories that $\theta_Y=\zeta_3$, and $\theta_X\in\{-1,\zeta_4,\zeta_8\}$.  Note that $\mathcal{C}$ is integral and solvable by \cite[Theorem 3.9]{MR3486174} and \cite[Theorem 1.6]{solvable}, and therefore Witt equivalent to a pointed modular fusion category \cite[Theorem 1.1]{MR3770935}, i.e.\ there exists a maximal connected \'etale algebra $A$ such that $\mathcal{C}_A^0$ is pointed, completely anisotropic, and nontrivial since $\xi\neq1$.  This forces $\mathcal{C}_A^0\simeq\mathcal{P}_2\boxtimes\mathcal{P}_3$ to be a braided equivalence where $\mathcal{P}_2$ is pointed with dimension a (possibly trivial) power of $2$ and $\mathcal{P}_3$ pointed with dimension a (possibly trivial) power of $3$.  The twists of $\mathcal{C}_A^0$ are inherited from $\mathcal{C}$ and are coprime, thus $\mathcal{P}_2$ or $\mathcal{P}_3$ is trivial.

\par Assume first that only $\mathcal{P}_2$ is trivial.  Then $\mathcal{P}_3\simeq\mathcal{C}(C_3,q)$ where $q(1)=\zeta_3$ by our previous assumptions.  This implies $\xi$, and hence $\gamma$, is $\zeta_4$, thus $t_\mathbbm{1}=\zeta_4^3$, $t_Y=\zeta_{12}$, and $\theta_X\in\{1,\zeta_4,\zeta_8^7\}$.  Therefore any irreducible summand of $\rho_\mathcal{C}'$ of level $12$ would be one-dimensional.  This violates the non-empty intersection criterion \cite[Lemma 3.18]{MR3486174}, so no such categories exist.

\par Secondly, if only $\mathcal{P}_3$ is trivial, $\mathcal{C}_A^0$ is braided equivalent to the unique metric group of order $4$ such that all nontrivial invertible objects have twist $-1$, or $\mathcal{C}(C_2,q)$ with $q(1)=\zeta_4$.  In the former case, the twists of $\mathcal{C}$ would be $1$, $-1$, and $\zeta_3$, but such a modular fusion category cannot have $\xi=-1$.  In the latter case, $\gamma=\zeta_8$, hence $t_\mathbbm{1}=\zeta_8^7$, $t_X=\zeta_8$, and $t_Y=\zeta_{24}^5$.  Such a combination of $t$-eigenvalues is not possible since any level $24$ irreducible summand of $\rho_\mathcal{C}'$ would be a tensor product of level $3$ and level $8$ irreducible factors, and all such products have $t$-eigenvalues which are not coprime by Lemmas \ref{one} and \ref{two}.
\end{proof}


\subsection{Three twists}\label{secsecthree}

\par The entirety of this subsection will prove the following classification.  For brevity, we will use the Lie-theoretic notation for the modular fusion categories of rank $2$ and $3$.  The interested reader may reference \cite{MR4079742} for a digestible overview of this family of categories.

\begin{theorem}\label{th:thre}
Let $\mathcal{C}$ be a modular fusion category with three distinct twists: $1$, $\theta$, and $\eta$.  Then $N=3$ or $\mathcal{C}$ is braided equivalent to one of categories recorded in Figure \ref{fig:three}.
\end{theorem}

\begin{figure}[H]
\centering
\begin{equation*}
\begin{array}{|c|ccccc|}
\hline \mathcal{C} & \#/\simeq &\textnormal{Indexed by} & \mathrm{FPdim}(\mathcal{C}) & \theta,\eta  & N\\\hline\hline
\mathcal{C}(C_2,q)\boxtimes\mathcal{C}(C_2,\overline{q}) & 1 & q:C_2\to\mathbb{C}^\times & 4 & \zeta_4,\zeta_4^3 & 4 \\
\mathcal{C}(C_4,q) & 6 & q:C_4\to\mathbb{C}^\times\text{ with} & 4 & -1,q(1) & 8 \\
 & & q(1)^8=1,q(1)^2\neq1 &  & & \\
\mathcal{I}_q & 8 & \textnormal{primitive } q^{16}=1 & 4 & -1,q  & 16 \\
\mathcal{C}(C_5,q) & 2 & q:C_5\to\mathbb{C}^\times & 5 & q(1),q(1)^{-1} & 5  \\
\mathcal{C}(\mathfrak{sl}_2,5,q)_\mathrm{ad}\boxtimes\mathcal{C}(\mathfrak{sl}_2,5,q^{-1})_\mathrm{ad} & 2 & \text{primitive }q^{10}=1 & \frac{5}{2}(3+\sqrt{5}) & q^4,q^{-4} & 5 \\
\mathcal{C}(\mathfrak{sl}_2,7,q)_\mathrm{ad} & 6 & \textnormal{primitive } q^{14}=1 & \frac{7}{4}\csc^2(\pi/7) & q^4,q^{12} & 7 \\
\mathcal{C}(C_4^2,q) & 1 & q:C_4^2\to\mathbb{C}^\times & 16 & \zeta_4,\zeta_4^3 & 4 \\
 & & q(x)\neq-1\text{ for }x\in C_4^2 & & & \\
\mathcal{Z}(\mathrm{Vec}_{C_2^3}^\omega) & 1 & \omega\in H^3(C_2^3,\mathbb{C}^\times) & 64 & \zeta_4,\zeta_4^3 & 4 \\
\hline
\end{array}
\end{equation*}
    \caption{Modular fusion categories with three distinct twists and $N\neq3$ up to braided eq.}%
    \label{fig:three}%
\end{figure}

Assume $\mathcal{C}$ is a modular fusion category with exactly three twists $1$, $\theta$, or $\eta$.  By Galois conjugacy of simple objects (see Appendix \ref{appa}), the $t$-eigenvalues of $\mathcal{C}$ are roots of unity of order dividing $2^4\cdot3^2\cdot5\cdot7$ by inspecting prime powers $p^k$ such that $\varphi(p^k)=(1/2)p^{k-1}(p-1)\leq3$.  But if $9$ divides the order of $t$, $\rho'_\mathcal{C}$ contains an irreducible summand $\nu$ of level divisible by $9$, hence an irreducible factor of $\nu$ is an irreducible $\mathrm{SL}(2,\mathbb{Z}/9\mathbb{Z})$-representation.  No such irreducible representation has $\leq3$ distinct $t$-eigenvalues by Lemmas \ref{one} and \ref{two}, so this case need not be considered, reducing the orders of $t$ under consideration to divisors of $2^4\cdot3\cdot5\cdot7$.  Lemma \ref{atemf} will show that $8$ dividing $n$ implies $N\in\{4,8,16\}$, and Lemma \ref{twelvemf} will show that $12$ dividing $n$ implies $N\in\{3,4,6\}$.  Sections \ref{7sec} and \ref{5sec} show that if $5$ or $7$ divide $n$, then $N$ is $5$ or $7$, respectively.  Since $N$ divides $n$ in generality \cite[Theorem 2.9(a)]{MR3486174}, this demonstrates our argument can be separated into the specific $N$ below, listed in the order of the subsection discussing each case with the number of braided equivalence classes of categories existing in each case.

\begin{figure}[H]
\centering
\begin{equation*}
\begin{array}{|ccc|}
\hline\textnormal{Order of $T$}& \textnormal{Location} & \#/\simeq \\\hline
3 & \textnormal{Section }\ref{3sec} & \infty \\
4 & \textnormal{Section }\ref{4sec} & 5 \\
5 & \textnormal{Section }\ref{5sec}& 4 \\
6 & \textnormal{Section }\ref{6sec}& 0 \\
7 & \textnormal{Section }\ref{7sec} & 6 \\
8 & \textnormal{Section }\ref{8sec}& 4 \\
16 & \textnormal{Lemma \ref{atemf}} & 8\\\hline
\end{array}
\end{equation*}
    \caption{Possible $N$, the section or result which classifies such modular fusion categories, and the corresponding number of braided equivalence classes of categories}%
    \label{fig:eightthree}%
\end{figure}

\par In addition to the more involved Lemmas \ref{atemf} and \ref{twelvemf}, we will rely on the following quadratic relation analogous to Equation (\ref{eqx}) to prove Theorem \ref{th:thre}.

\begin{lemma}\label{lan5}
Let $\mathcal{C}$ be a modular fusion category with three distinct twists: $1$, $\theta$, and $\eta$.  Then
\begin{equation}\label{theeq}
D_1+D_\theta+D_\eta=D_1^2+D_\theta^2+D_\eta^2+(\theta+\theta^{-1})D_1D_\theta+(\eta+\eta^{-1})D_1D_\eta+(\theta\eta^{-1}+\theta^{-1}\eta)D_\theta D_\eta.
\end{equation}
\end{lemma}

\begin{proof}
Expand $D=\tau_1\tau_{-1}$ \cite[Proposition 8.15.4]{tcat} in terms of $D_1,D_\theta,D_\eta$.
\end{proof}

\begin{lemma}\label{atemf}
Let $\mathcal{C}$ be a modular fusion category with three distinct twists: $1$, $\theta$, and $\eta$.  If $8$ divides $n$, then either
\begin{enumerate}
\item $N=16$ and $\mathcal{C}$ is an Ising modular fusion category \cite[Appendix B]{DGNO},
\item $N=8$, $\theta=-1$ and $\eta$ is a primitive $8$th root of unity without loss of generality, or
\item $N=4$.
\end{enumerate}
\end{lemma}

\begin{proof}
\par Eight must divide the level of some irreducible summand $\nu\subset\rho_\mathcal{C}'$, hence $\nu\cong\nu_0\otimes\nu_2$ where $\nu_0$ is an irreducible factor with level coprime to $2$ and $\nu_2$ is an irreducible factor with level divisible by $8$.  Note that all possible $\nu_2$ have at least two distinct $t$-eigenvalues by Lemma \ref{one}, therefore $\nu_0$ must be one-dimensional, say with unique $t$-eigenvalue $\omega$ for some third root of unity $\omega$.  Lemmas \ref{one} and \ref{two} imply that $\nu_2$ is either level $16$ or level $8$ and we will treat these cases separately.

\par If $\nu_2$ is level $16$, then the $t$-eigenvalues of $\nu$ are three distinct primitive $8$th and $16$th roots of unity, hence the $t$-eigenvalues of $\nu$ are three distinct $24$th and $48$th roots of unity, determined by $\nu_0$ and $\nu_2$.  Since $\nu$ already has three distinct $t$-eigenvalues, the $t$-eigenvalues of any other irreducible summand of $\rho_\mathcal{C}'$ must be a subset of those of $\nu$.  Therefore it has level divisible by $24$ as well, and moreover must be isomorphic to $\nu$ itself by the above reasoning.  Thus $\rho_\mathcal{C}'\cong n\nu$ for some $n\in\mathbb{Z}_{\geq1}$, and moreover $n=1$ and $\mathrm{rank}(\mathcal{C})=3$ by \cite[Lemma 5.2.2]{MR4834521}.  There are exactly $8$ braided equivalence classes of modular fusion categories of rank $3$ with a $T$-matrix of even order, each with two distinct spherical structures.  These $16$ are known as the Ising modular fusion categories \cite[Appendix B]{DGNO}, and are denoted $\mathcal{I}_q$ where $q=\eta$ is a primitive $16$th root of unity.

\par Alternatively, $\nu_2$ is level $8$ and has either two or three distinct $t$-eigenvalues by Lemmas \ref{one} and \ref{two} and we will treat these cases separately to finish the proof.  If $\nu_2$ has three distinct $t$-eigenvalues, then $\nu_0$ is one-dimensional, hence $\nu$ has $t$-eigenvalues $\{\omega\zeta,\omega\zeta_8,\omega\zeta_8^5\}$ or $\{\omega\zeta,\omega\zeta_8^3,\omega\zeta_8^7\}$ where $\zeta$ is a fourth root of unity and $\omega$ is a third root of unity; these possibilities can be treated identically as $\omega$ is divided out when computing the twists.  That is to say for each possible $t$-spectrum of $\nu_2$, there are three possible $\gamma=t_\mathbbm{1}^{-1}$, giving ostensibly $24$ options for the sets of twists of $\mathcal{C}$, illustrated in Figure \ref{fig:eightthreep}.

\begin{figure}[H]
\centering
\begin{equation*}
\begin{array}{|c|c|c|c||c|c|c|c|}
\hline t\text{-spectrum} & \gamma& \xi & \text{Twists} & t\text{-spectrum} & \gamma & \xi & \text{Twists} \\\hline
\omega,\omega\zeta_8,\omega\zeta_8^5 & \omega^{-1} & 1 & 1,\zeta_8,\zeta_8^5 & \omega\zeta_4,\omega\zeta_8,\omega\zeta_8^5 & \omega^{-1}\zeta_4^3 & \zeta_4 & 1,\zeta_8^3,\zeta_8^7\\
 & \omega^{-1}\zeta_8^7 & \zeta_8^5 &  1,-1,\zeta_8^7 & & \omega^{-1}\zeta_8^7 & \zeta_8^5 & 1,-1,\zeta_8 \\
 & \omega^{-1}\zeta_8^3 & \zeta_8 & 1,-1,\zeta_8^3 &  & \omega^{-1}\zeta_8^3 & \zeta_8 & 1,-1,\zeta_8^5 \\
-\omega,\omega\zeta_8,\omega\zeta_8^5 & -\omega^{-1} & -1 & 1,\zeta_8,\zeta_8^5 &  \omega\zeta_4^3,\omega\zeta_8,\omega\zeta_8^5 & \omega^{-1}\zeta_4 & \zeta_4^3 & 1,\zeta_8^3,\zeta_8^7  \\
 & \omega^{-1}\zeta_8^7 & \zeta_8^5& 1,-1,\zeta_8  & & \omega^{-1}\zeta_8^7 &\zeta_8^5 & 1,-1,\zeta_8^5  \\
 & \omega^{-1}\zeta_8^3 & \zeta_8 & 1,-1,\zeta_8^5 &  & \omega^{-1}\zeta_8^3 & \zeta_8& 1,-1,\zeta_8  \\\hline\hline
\omega,\omega\zeta_8^3,\omega\zeta_8^7  &\omega^{-1} & 1 & 1, \zeta_8^3,\zeta_8^7&\omega\zeta_4,\omega\zeta_8^3,\omega\zeta_8^7 & \omega^{-1}\zeta_4^3& \zeta_4& 1,\zeta_8,\zeta_8^5\\
 & \omega^{-1}\zeta_8^5 & \zeta_8^7& 1,-1,\zeta_8^5 & &\omega^{-1}\zeta_8^5 & \zeta_8^7& 1,-1,\zeta_8^7\\
 & \omega^{-1}\zeta_8 & \zeta_8^3& 1,-1,\zeta_8 & & \omega^{-1}\zeta_8& \zeta_8^3& 1,-1,\zeta_8^3\\
-\omega,\omega\zeta_8^3,\omega\zeta_8^7& -\omega^{-1} & -1 & 1, \zeta_8^3,\zeta_8^7 & \omega\zeta_4^3,\omega\zeta_8^3,\omega\zeta_8^7 & \omega^{-1}\zeta_4 &\zeta_4^3 & 1,\zeta_8,\zeta_8^5\\
& \omega^{-1}\zeta_8^5& \zeta_8^7& 1,-1,\zeta_8& & \omega^{-1}\zeta_8^5& \zeta_8^7& 1,-1,\zeta_8^3\\ 
&\omega^{-1}\zeta_8 & \zeta_8^3&1,-1,\zeta_8^5 & & \omega^{-1}\zeta_8 & \zeta_8^3& 1,-1,\zeta_8^7\\\hline
\end{array}
\end{equation*}
    \caption{Twists when $\nu_2$ is level $8$ with three distinct $t$-eigenvalues}%
    \label{fig:eightthreep}%
\end{figure}

\par Many of these combinations of fulls twists are incompatible with the corresponding $\xi$ since $\xi$ must have the same argument as $\tau_1=D_1+D_\theta\theta+D_\eta\eta$.  The possible twists and $\xi$ when $\theta$ is a primitive eighth root of unity are the Galois conjugate pair $\{1,\zeta_8,\zeta_8^5\}$ and $\{1,\zeta_8^3,\zeta_8^7\}$, both with $\xi=1$, and the Galois conjugate pair $\{1,\zeta_8,\zeta_8^5\}$ with $\xi=\zeta_4^3$ and $\{1,\zeta_8^3,\zeta_8^7\}$ with $\xi=\zeta_4$.  In the former case, $D_\theta=D_\eta$, hence,
\begin{equation}
D_1+2D_\theta=D_1^2+2D_\theta^2+(\theta+\theta^{-1})D_1D_\theta-(\theta+\theta^{-1})D_1D_\theta-2D_\theta^2
\end{equation}
by Lemma \ref{lan5}, which implies $D_\theta=(1/2)D_1(D_1-1)$.  If $D_1\in\mathbb{Z}$, then $D=D_1^2$ is a perfect square integer with $\xi=1$.  As $\mathcal{C}$ is nilpotent \cite[Theorem 8.28]{ENO}, $\mathcal{C}$ is weakly group-theoretical, and therefore $\mathcal{C}$ is Witt equivalent to a pointed modular fusion category \cite[Theorem 1.1]{MR3770935} based on the possible twists.  Furthermore, $\mathcal{C}$ must be Witt trivial as there are no metric groups with only twists $1$, $\zeta_8$, and $\zeta_8^5$.  So we have $\mathcal{C}$ is the twisted double of a $2$-group \cite[Theorem 6.6]{DGNO} of order $D_1$.  But let $g\in G$ be any nontrivial element of order $2$.  Then $I(g)$ satisfies the zero trace property (Equation (\ref{zerotrace})) only if $\theta_Y\in\{\zeta_8,\zeta_8^5\}$ for all simple $Y\subset I(g)$.  Moreover, $I(g)$ fails the square trace property (Equation (\ref{zero2trace})) since $\theta_Y^2=\zeta_4$ for all $Y\subset I(g)$.  Therefore no such categories exist with $D\in\mathbb{Z}$, twists $\{1,\zeta_8,\zeta_8^5\}$, and $\xi=1$.

\par If $D_1\not\in\mathbb{Z}$, $D_\theta$ is an algebraic $d$-number as $\mathcal{C}_\theta$ is closed under Galois conjugacy.  Moreover $D_\theta^2$ is an integer multiple of an algebraic unit, which implies $(D_1-1)^2$ is a rational multiple of a unit, i.e.\ $D_1-1$ is an algebraic $d$-number since $D_1-1$ is an algebraic integer.  Then Proposition \ref{dnumberz} implies that $D_1$ has a Galois conjugate less than $2$.  But since $D\not\in\mathbb{Z}$, the tensor unit has a Galois conjugate object (see Appendix \ref{appa}), say $X$, such that $\dim(X)^2=D/\sigma(D)$ where $\sigma\in\mathrm{Gal}(\overline{\mathbb{Q}}/\mathbb{Q})$ satisfies $\sigma(\sqrt{2})=-\sqrt{2}$, and we have either $1+D/\sigma(D)$ or $1+\sigma(D)/D$ is greater than $2$ since $D/\sigma(D)$ is a totally positive algebraic integer.  Therefore no such categories exist with twists $\{1,\zeta_8,\zeta_8^5\}$ or $\{1,\zeta_8^3,\zeta_8^7\}$.

\par To complete the case when $\theta$ is a primitive $8$th root of unity, it remains to consider twists $\{1,\zeta_8,\zeta_8^5\}$ with $\xi=\zeta_4^3$ without loss of generality.
\begin{figure}[H]
\centering
\begin{equation*}
\begin{tikzpicture}[scale=0.18]
\draw[<->,dotted] (0,-10) -- (0,10);
\draw[<->,dotted] (-10,0) -- (10,0);
\draw[thick] (0,0) -- node[above left] {$D_\eta+D_\theta$} (-6,-6);
\draw[->,dashed] (0,0) -- (10,10);
\draw[->,dashed] (0,0) -- (-10,-10);
\draw[-,thick] (0,0) -- node[right] {$\sqrt{D}$} (0,-6);
\draw[thick] (-6,-6) -- node[below] {$D_1$} (0,-6);
\end{tikzpicture}
\end{equation*}
    \caption{A geometric view of $\tau_1=\zeta_4^3\sqrt{D}$ with twists $\{1,\zeta_8,\zeta_8^5\}$ and $\xi=\zeta_4^3$}%
    \label{fig:eightonenowzz}%
\end{figure}
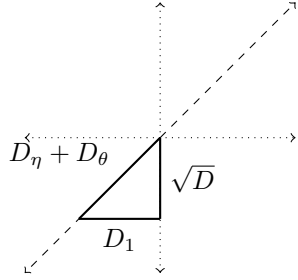
Geometrically, $D_1=\sqrt{D}=\sqrt{D_1+D_\theta+D_\eta}$, hence $D_1=(1/2)\left(1+\sqrt{1+4(D_\theta+D_\eta)}\right)$.  Elementary trigonometry gives $D_1\sqrt{2}=D_\theta+D_\eta$, hence $D_1=(1/2)\left(1+\sqrt{1+4\sqrt{2}D_1}\right)$ which implies $D_1=1+\sqrt{2}$ which cannot occur since $D_1$ has a Galois conjugate less than zero, and $D_1$ must be totally greater than or equal to $1$ since $\theta_\mathbbm{1}=1$.  We may then conclude $\theta=-1$ and $\eta$ is a primitive $8$th root of unity, as we aimed to prove.

\par If $\nu_2$ has only two distinct $t$-eigenvalues, then $\nu_2$ is one of the four isomorphism classes of level $8$ irreducible representations of dimension $2$.  This forces all $t$-eigenvalues of $\mathcal{C}$ to be $24$th roots of unity, i.e.\ $\{\omega\zeta_8,\omega\zeta_8^3,\omega\zeta_8^5\}$, $\{\omega\zeta_8,\omega\zeta_8^3,\omega\zeta_8^7\}$, $\{\omega\zeta_8,\omega\zeta_8^5,\omega\zeta_8^7\}$, or $\{\omega\zeta_8^3,\omega\zeta_8^5,\omega\zeta_8^7\}$ for some primitive third root of unity $\omega$.  For each option, there are three choices for $\gamma$, giving ostensibly $12$ options for the sets of twists of $\mathcal{C}$, recorded in Figure \ref{fig:eighttwo}.  All such sets of twists have $N=4$ as we aimed to prove.

\begin{figure}[H]
\centering
\begin{equation*}
\begin{array}{|c|c|c|c||c|c|c|c|}
\hline t-\text{eigenvalues} & \gamma & \xi &  \text{Twists} & t-\text{eigenvalues} & \gamma & \xi &\text{Twists}\\\hline
\omega\zeta_8,\omega\zeta_8^3,\omega\zeta_8^5&\omega^{-1}\zeta_8^7& \zeta_8^5 & 1,-1,\zeta_4  & \omega\zeta_8,\omega\zeta_8^5,\omega\zeta_8^7&\omega^{-1}\zeta_8^7& \zeta_8^5 &1,-1,\zeta_4^3 \\
&\omega^{-1}\zeta_8^5& \zeta_8^7&1,\zeta_4,\zeta_4^3 && \omega^{-1}\zeta_8^3& \zeta_8&1,-1,\zeta_4 \\
&\omega^{-1}\zeta_8^3& \zeta_8&1,-1,\zeta_4^3 & &\omega^{-1}\zeta_8& \zeta_8^3&1,\zeta_4,\zeta_4^3 \\\hline\hline
\omega\zeta_8,\omega\zeta_8^3,\omega\zeta_8^7&\omega^{-1}\zeta_8^7& \zeta_8^5 &1,\zeta_4,\zeta_4^3 & \omega\zeta_8^3,\omega\zeta_8^5,\omega\zeta_8^7&\omega^{-1}\zeta_8^5& \zeta_8^7&1,-1,\zeta_4 \\
&\omega^{-1}\zeta_8^5& \zeta_8^7&1,-1,\zeta_4^3 & &\omega^{-1}\zeta_8^3& \zeta_8&1,\zeta_4,\zeta_4^3\\
&\omega^{-1}\zeta_8& \zeta_8^3&1,-1,\zeta_4 & &\omega^{-1}\zeta_8& \zeta_8^3 &1,-1,\zeta_4^3 \\\hline
\end{array}
\end{equation*}
    \caption{Twists when $\nu_2$ is level $8$ with two distinct $t$-eigenvalues}%
    \label{fig:eighttwo}%
\end{figure}
\end{proof}

\begin{lemma}\label{twelvemf}
Let $\mathcal{C}$ be a modular fusion category with three distinct twists: $1$, $\theta$, and $\eta$.  If $n=12$, then $N\in\{3,4,6\}$.
\end{lemma}

\begin{proof}
As in previous proofs, Lemmas \ref{one} and \ref{two} give a finite number of possible irreducible summands of $\rho_\mathcal{C}'$, all of dimensions $1$, $2$, or $3$.  Note that if $\rho_\mathcal{C}'$ has a unique isomorphism class of irreducible $2$-dimensional summands with $t$-spectrum $\{\xi_1,\xi_2\}$ for some roots of unity $\xi_1,\xi_2$, and no irreducible $3$-dimensional summands, then $\rho_\mathcal{C}'\cong \nu\oplus\nu'$ where $\nu$ is the direct sum of all irreducible summands of $\rho_\mathcal{C}'$ whose $t$-spectra are a subset of $\{\xi_1,\xi_2\}$, and $\nu'$ is the complement of $\nu$ in $\rho_\mathcal{C}'$.  This would violate \cite[Lemma 3.18]{MR3632091}, hence there are only two cases to consider: the irreducible summands of $\rho_\mathcal{C}'$ include only two distinct isomorphism classes of $2$-dimensional representations and no $3$-dimensional representations, or $\rho_\mathcal{C}'$ has a $3$-dimensional irreducible summand.  To understand the former case, we record all possible $2$-dimensional irreducible summands in Figure \ref{fig:twelvewan}.
\begin{figure}[H]
\centering
\begin{equation*}
\begin{array}{|cc|cc|cc|}\hline
\textnormal{Name} & t-\mathrm{spectrum} & \textnormal{Name} & t-\mathrm{spectrum} & \textnormal{Name} & t-\mathrm{spectrum}\\\hline
N_1(\chi)\otimes C_5 & \zeta_4^3,\zeta_{12} & N_1(\chi)\otimes C_7 & \zeta_4,\zeta_{12}^{11}       &  N_1(\chi)\otimes C_{10} & -1,\zeta_6 \\
N_1(\chi)\otimes C_9 & \zeta_{12},\zeta_{12}^5  & N_1(\chi)\otimes C_3 & \zeta_{12}^7,\zeta_{12}^{11} & N_1(\chi)\otimes C_2 & -1,\zeta_6^5 \\
N_2(\chi)\otimes C_4 & \zeta_{12},\zeta_{12}^7 & N_2(\chi) & \zeta_4,\zeta_4^3 &  N_1(\chi) & \zeta_3,\zeta_3^2\\
N_2(\chi)\otimes C_8 & \zeta_{12}^5,\zeta_{12}^{11}  & N_1(\chi)\otimes C_8 & 1,\zeta_3 &  N_1(\chi_1)\otimes C_4 & \zeta_3,\zeta_6^5 \\
N_1(\chi)\otimes C_{11} & \zeta_4,\zeta_{12}^7  &  N_1(\chi)\otimes C_4 & 1,\zeta_3^2 & N_1(\chi_1)\otimes C_8 & \zeta_3^2,\zeta_6\\
N_1(\chi)\otimes C_1 & \zeta_4^3,\zeta_{12}^5  &  N_1(\chi_1) & -1,1 & N_1(\chi)\otimes C_6 & \zeta_6,\zeta_6^5 \\\hline
\end{array}
\end{equation*}
    \caption{Two-dimensional irreducible representations with level dividing $12$}%
    \label{fig:twelvewan}%
\end{figure}
We may assume by Galois conjugacy of modular fusion categories that $\zeta_{12}$ is a $t$-eigenvalue of $\mathcal{C}$, hence ostensibly, the possible $t$-spectra of $\mathcal{C}$ and corresponding twists are given in Figure \ref{fig:twelvetoo} up to Galois conjugates.  One observes that $N$ divides $6$ in any case.
\begin{figure}[H]
\centering
\begin{equation*}
\begin{array}{|cccc|}\hline
\rho_\mathcal{C}'\textnormal{ contains summand} & t-\mathrm{spectrum} & \gamma & \textnormal{Twists}\\\hline
N_2(\chi)\oplus N_1(\chi)\otimes C_5 & \zeta_4,\zeta_4^3,\zeta_{12} & \zeta_4^3 & 1,-1,\zeta_6^5\\
 &  & \zeta_4 & 1,-1,\zeta_3 \\
 &  & \zeta_{12}^{11} & 1,\zeta_3^2,\zeta_6,\\
N_1(\chi)\otimes(C_4\oplus C_{11})  & \zeta_4,\zeta_{12},\zeta_{12}^7& \zeta_4^3 & 1,\zeta_3,\zeta_6^5\\
 &  & \zeta_{12}^{11} & 1,-1,\zeta_6\\
 &  & \zeta_{12}^5 & 1,-1,\zeta_3^2\\
N_1(\chi)\otimes(C_i\oplus C_j) & \zeta_4^3,\zeta_{12},\zeta_{12}^5 & \zeta_4 & 1,\zeta_3,\zeta_3^2\\
i\neq j\in\{1,5,9\} &  & \zeta_{12}^{11}  & 1,\zeta_3,\zeta_3^2 \\
 &  & \zeta_{12}^7 & 1,\zeta_3,\zeta_3^2 \\
N_1(\chi)\otimes(C_5\oplus C_4) & \zeta_4^3,\zeta_{12},\zeta_{12}^7 & \zeta_4 & 1,\zeta_3,\zeta_6^5\\
 &  &  \zeta_{12}^{11}& 1,-1,\zeta_3^2\\
 &  & \zeta_{12}^5 & 1,-1,\zeta_6 \\
N_1(\chi)\otimes(C_9\oplus C_4)  & \zeta_{12},\zeta_{12}^5,\zeta_{12}^7& \zeta_{12}^{11} & 1,-1,\zeta_3\\
 &  & \zeta_{12}^7 & 1,\zeta_3^2,\zeta_6 \\
 &  & \zeta_{12}^5 & 1,-1,\zeta_6^5\\
N_1(\chi)\otimes(C_9\oplus C_8)  & \zeta_{12},\zeta_{12}^5,\zeta_{12}^{11}& \zeta_{12}^{11} & 1,\zeta_3,\zeta_6^5\\
 &  & \zeta_{12}^7 & 1,-1,\zeta_3^2\\
 &  &\zeta_{12}  & 1,-1,\zeta_6\\\hline
\end{array}
\end{equation*}
    \caption{Possible twists when $12$ divides the order of $t$}%
    \label{fig:twelvetoo}%
\end{figure}
Lastly we assume there exists an irreducible summand $\nu$ of $\rho_\mathcal{C}'$ which is $3$-dimensional, and thus there exists a factorization $\nu\cong\nu_2\otimes\nu_3$ where $\nu_2$ has level $4$ and $\nu_3$ has level $3$.  If $\nu_2$ is $3$-dimensional, then $\nu_3$ must have a unique primitive third root of unity as its $t$-spectrum, hence scaling by $\gamma^{-1}$ implies the twists all have order $4$.  The same argument holds if $\nu_3$ is $3$-dimensional, forcing the order of $T$ to be $3$.
\end{proof}

This completes the preliminary arguments and we may now proceed to classify the categories in accordance to Figure \ref{fig:eightthree}.

\subsubsection{$N=3$}\label{3sec}

\par This case lies firmly in the study of finite groups.  In this case $\mathcal{C}$ is integral, and therefore nilpotent and $D=3^m$ for some $m\in\mathbb{Z}_{\geq1}$.  Moreover, $\mathcal{C}\boxtimes\mathcal{P}\simeq\mathcal{Z}(\mathrm{Vec}_G^\omega)$ is a braided equivalence for a pointed modular fusion category $\mathcal{P}$ of dimension $1$, $3$, or $9$, a finite group $G$ of exponent $3$, and $\omega\in H^3(G,\mathbb{C}^\times)$.  Infinitely many examples exist, e.g.\ $N=3$ for $\mathcal{Z}(\mathrm{Vec}_G^\omega)$ for any finite group $G$ of exponent $3$ when $\omega$ is trivial \cite[Theorem 9.2]{MR2313527}.


\subsubsection{$N=7$}\label{7sec}

\par Assume $7$ divides $N$.  In this case $\rho'_\mathcal{C}$ contains a simple summand $\nu$ of level divisible by $7$, hence a factor of $\nu$ is an irreducible $\mathrm{SL}(2,\mathbb{Z}/7\mathbb{Z})$-representation by Lemmas \ref{one} and \ref{two}; there are exactly two isomorphism classes of such representations with $\leq3$ distinct $t$-eigenvalues, each having exactly $3$ distinct $t$-eigenvalues.  Thus if $\nu$ contains any other irreducible factors, they are one-dimensional.  Therefore, up to Galois conjugacy of modular fusion categories, the $t$-spectrum of $\nu$ is $\{\epsilon\zeta_7,\epsilon\zeta_7^2,\epsilon\zeta_7^4\}$ where $\epsilon^{12}=1$.  If $\nu'\subset\rho_\mathcal{C}'$ is any other irreducible summand, then $\nu'$ must share a $t$-eigenvalue with $\nu$ and moreover $\nu'=\nu$ by the same reasoning and thus $\rho_\mathcal{C}'\cong m\nu$ for some $m\in\mathbb{Z}_{\geq1}$.  By \cite[Lemma 5.2.2]{MR4834521}, $m=1$ and $\mathrm{rank}(\mathcal{C})=3$.  Moreover, $\mathcal{C}\simeq\mathcal{C}(\mathfrak{sl}_2,7,q)_\mathrm{ad}$ where $q^2$ is a primitive seventh root of unity by \cite[Theorem 1.1]{MR3427429}.


\subsubsection{$N=6$}\label{6sec}

\par In this case, $\mathcal{C}$ is integral \cite[Theorem 2.7]{MR3486174}, so $\xi^8=1$ \cite[Proposition 2.6]{MR4836055}, and thus we can assume $\gamma^8=1$ as well.  The $t$-eigenvalues are then $\gamma^{-1}$, $\theta\gamma^{-1}$, and $\eta\gamma^{-1}$ where either both $\theta$ and $\eta$ have order divisible by $3$, or just $\eta$ without loss of generality.  If $\gamma$ were a primitive eighth root of unity, then there exists an irreducible summand $\nu\subset\rho_\mathcal{C}'$ of level $24$.  Lemmas \ref{one} and \ref{two} imply that $\nu\cong\nu_2\otimes\nu_3$ where $\nu_2$ is an irreducible factor of level $8$ and dimension $2$ or $3$, and $\nu_3$ is an irreducible factor of level $3$ of dimension $1$.  Therefore the $t$-spectrum of $\nu$ is either two distinct $24$th roots of unity, or two distinct $24$th roots of unity and a root of unity whose order is divisible by $3$.  The latter case is impossible since $\mathcal{C}$ has a $t$-eigenvalue of order $8$ and the former case is impossible since any other irreducible summand $\nu'\subset\rho_\mathcal{C}'$ which has a $t$-eigenvalue of order $8$, has at least two distinct $t$-eigenvalues by Lemmas \ref{one} and \ref{two}.  Moreover $\gamma^4=\xi^4=1$.

\par Ostensibly, there are $20$ possibilities for distinct, nontrivial twists $\theta$ and $\eta$ which are sixth roots of unity, but we have assumed $N=6$ and Theorem \ref{thmthreetwist} implies they cannot have coprime order.  This leaves only the options $\{1,-1,\zeta_6^{\pm1}\}$, $\{1,\zeta_3,\zeta_6^{\pm1}\}$, $\{1,\zeta_3^2,\zeta_6^{\pm1}\}$, or $\{1,\zeta_6,\zeta_6^5\}$, while up to complex conjugacy of modular fusion categories, we need only consider the pairs of cases $\{1,-1,\zeta_6\}$, $\{1,\zeta_6,\zeta_6^5\}$, and $\{1,\zeta_3,\zeta_6\}$, $\{1,\zeta_3,\zeta_6^5\}$.

\par For $\{1,-1,\zeta_6\}$, geometrically we must have $\xi=\zeta_4$.  As $\mathcal{C}$ is integral, $\mathcal{C}$ is solvable \cite[Theorem 1.6]{solvable}, hence weakly group-theoretical, and moreover Witt equivalent to a pointed modular fusion category $\mathcal{P}$ with $\xi=\zeta_4$ \cite[Theorem 1.1]{MR3770935}.  As $\mathcal{P}$ is nontrival, $\mathcal{P}\simeq\mathcal{P}_2\boxtimes\mathcal{P}_3$ is a braided equivalence where $\mathcal{P}_2$ is a metric $2$-group and $\mathcal{P}_3$ is a metric $3$-group.  If $\mathcal{P}_3$ were trivial, then the twists of $\mathcal{P}\simeq\mathcal{P}_2$ are only $\pm1$, violating the fact that $\xi=\zeta_4$.  Hence $\mathcal{P}_3$ is nontrivial which forces $\mathcal{P}_2$ to be nontrivial as well since $\mathcal{C}$ has no twists which are third roots of unity.  Checking the very small number of options, one finds there are no such metric groups with the correct first multiplicative central charge and twists.  When the twists are $\{1,\zeta_6,\zeta_6^5\}$, then $\gamma=\xi=1$.  There are no nontrivial pointed products $\mathcal{P}\simeq\mathcal{P}_2\boxtimes\mathcal{P}_3$ which $\mathcal{C}$ could be Witt equivalent to, hence $\mathcal{C}\simeq\mathcal{Z}(\mathrm{Vec}_G^\omega)$ for a finite group $G$ and $\omega\in H^3(G,\mathbb{C}^\times)$.  But no nontrivial positive integer linear combination of $1,\zeta_6,\zeta_6^5$ can vanish, hence $I(g)$ for any nontrivial $g\in G$ cannot satisfy Equation (\ref{zerotrace}).

\par For $\{1,\zeta_3,\zeta_6\}$, we must have $\gamma=\xi=\zeta_4$, hence $\mathcal{C}$ is Witt equivalent to a nontrivial pointed product $\mathcal{P}\simeq\mathcal{P}_2\boxtimes\mathcal{P}_3$ as in the argument in the preceding paragraph; no such products have the correct twists and $\xi$.  Lastly, with twists $\{1,\zeta_3,\zeta_6^5\}$ then $\xi=1$ or $\xi=\zeta_4$.  There is a unique product of nontrivial pointed categories with the desired twists: $\mathcal{C}(C_2^2,q_{-1})\boxtimes\mathcal{C}(C_3,q_{\zeta_3})$ where $q_{-1}$ and $q_{\zeta_3}$ take the subscripted values on all nontrivial elements of $C_2^2$ and $C_3$, respectively.  But this product has central charge $\zeta_4^3$, so it cannot be Witt equivalent to $\mathcal{C}$, and when $\xi=1$, $\mathcal{C}$ is a twisted double of a finite group $G$ with a nontrivial element of order $2$, say $g$.  It is possible that $I(g)$ satisfies the zero trace property (Equation (\ref{zerotrace})), but only if $\theta_Y\in\{\pm\zeta_3\}$ for all $Y\subset I(g)$.  But in this case $I(g)$ fails the square trace property (Equation (\ref{zero2trace})) for self dual objects since $\theta_Y^2=\zeta_3^2$ for all $Y\subset I(g)$.


\subsubsection{$N=5$}\label{5sec}

When $N=5$, it is possible to characterize $\dim(\mathcal{C})$ as one of the $5$ algebraic $d$-numbers from Lemma \ref{lemroot5} in Appendix \ref{appa}.  All such modular fusion categories are easily described using the existing literature.

\begin{lemma}\label{propfib}
Let $\mathcal{C}$ be a modular fusion category.  If $D$ is one of the possible global dimensions in Lemma \ref{lemroot5}, then $\mathcal{C}$ is equivalent as a modular fusion category to one of the following:
\begin{enumerate}
\item the pointed modular fusion categories $\mathcal{C}(C_5,q)$,
\item the Fibonacci modular fusion categories $\mathcal{C}(\mathfrak{sl}_2,5,q)_\mathrm{ad}$, or
\item the products of Fibonacci modular fusion categories $\mathcal{C}(\mathfrak{sl}_2,5,q_1)_\mathrm{ad}\boxtimes\mathcal{C}(\mathfrak{sl}_2,5,q_2)_\mathrm{ad}$.
\end{enumerate}
\end{lemma}

\begin{proof}
The classifications for $\dim(\mathcal{C})\in\{(1/2)(3-\sqrt{5}),5,(1/2)(3+\sqrt{5})\}$ were completed in \cite[Example 5.1.2]{ostrikremarks} where fusion categories of these dimensions were described.  We will classify modular fusion categories of dimensions $(5/2)(3-\sqrt{5})$ which completes the argument by Galois conjugacy of modular fusion categories.  Note that the algebraic norm of $(5/2)(3-\sqrt{5})$ is $5^2$.  Let $\sigma\in\mathrm{Gal}(\overline{\mathbb{Q}}/\mathbb{Q})$ such that $\sigma(\sqrt{5})=-\sqrt{5}$.  Then $\dim(\mathcal{C})\boxtimes\dim(\mathcal{C}^\sigma)=5^2$ and our claim follows from \cite[Theorem 4.15]{MR4564492}.
\end{proof}

\par Assume now $N=5$.  In this case $\rho'_\mathcal{C}$ contains a simple summand $\nu$ of level divisible by $5$, hence a factor of this simple summand is an irreducible $\mathrm{SL}(2,\mathbb{Z}/5\mathbb{Z})$-representation.  There are four isomorphism classes of such representations with $\leq3$ distinct $t$-eigenvalues by Lemmas \ref{one} and \ref{two}: two are three-dimensional and have exactly $3$ distinct $t$-eigenvalues $1,\zeta,\zeta^{-1}$ for a primitive fifth root of unity $\zeta$; two are two-dimensional and have exactly $2$ distinct $t$-eigenvalues $\zeta,\zeta^{-1}$ for a primitive fifth root of unity $\zeta$.  Thus if $\nu$ contains any other irreducible factors, they are one-dimensional.  Therefore, up to Galois conjugacy of categories, the $t$-spectrum of $\mathcal{C}$ is $\{\epsilon,\epsilon\zeta_5,\epsilon\zeta_5^4\}$ where $\epsilon^{12}=1$.  Note that $\gamma^{-1}\in\{\epsilon,\epsilon\zeta_5,\epsilon\zeta_5^4\}$, so the twists are either $\{1,\zeta_5,\zeta_5^4\}$, $\{1,\zeta_5,\zeta_5^2\}$, or $\{1,\zeta_5^3,\zeta_5^4\}$.  The maximal real subfield of $\mathbb{Q}(\zeta_5)$ is $\mathbb{Q}(\sqrt{5})$, so this is where the dimensions of such a category lie \cite[Theorem 2.7]{MR3486174}.

\par In the first case, using Lemma \ref{lan5}, with $\varphi:=(1/2)(1+\sqrt{5})$ for brevity, we have
\begin{equation}
D_1+D_\theta+D_\eta=D_1^2+D_\theta^2+D_\eta^2+\varphi^{-1}D_1(D_\theta+D_\eta)+2D_\theta D_\eta.
\end{equation}
Factored as a quadratic relation in $D_1$, we have $0=D_1^2+(w\varphi^{-1}-1)D_1+w(w-1)$, where $w:=D_\theta+D_\eta$ for brevity, with solutions
\begin{equation}\label{xsol}
D_1=\frac{1}{2}\left(1-\varphi^{-1}w\pm\sqrt{(w\varphi^{-1}-1)^2-4w(w-1)}\right)
\end{equation}
For $D_1$ to be real we require $(w\varphi^{-1}-1)^2-4w(w-1)\geq0$, thus $(\varphi^{-2}-4)w^2+(4-2\varphi^{-1})w+1\geq0$.  As a quadratic in $w$, one may verify that the above being nonnegative implies $w<2$. and the solutions for $D_1$ in Equation (\ref{xsol}) are maximized when $D_1<2$ as well.  Moreover $D<4$.  Hence $D=(5/2)(3-\sqrt{5})\approx1.910$ by Lemma \ref{lemroot5} and our claim follows from Lemma \ref{propfib}.

\par In the second case, we have
\begin{equation}
D_1+D_\theta+D_\eta=D_1^2+D_\theta^2+D_\eta^2-\varphi D_1D_\eta+\varphi^{-1}D_\theta D_1+\varphi^{-1}D_\theta D_\eta.
\end{equation}
Factored as a quadratic relation in $D_1$, we have
\begin{equation}
0=D_1^2+(\varphi^{-1}D_\theta-\varphi D_\eta-1)D_1+D_\theta^2+D_\eta^2+\varphi^{-1}D_\theta D_\eta-D_\theta-D_\eta
\end{equation}
with solutions
\begin{equation}\label{xsol2}
D_1=\frac{1}{2}\left(1-\varphi^{-1}y+\varphi D_\eta\pm\sqrt{(\varphi^{-1}D_\theta-\varphi D_\eta-1)^2-4(D_\theta^2+D_\eta^2+\varphi^{-1}D_\theta D_\eta-D_\theta-D_\eta)}\right).
\end{equation}
One may verify the discriminant is nonnegative only if $D_\theta<2$ and $D_\eta<7$, and the solutions for $D_1$ in Equation (\ref{xsol2}) are maximized when $D_1<7$, hence $D$ is one of the numbers from Lemma \ref{lemroot5} and our claim follows from Lemma \ref{propfib}.

\par In the third case, we have
\begin{equation}
D_1+D_\theta+D_\eta=D_1^2+D_\theta^2+D_\eta^2-D_1(\varphi D_\theta+\varphi^{-1}D_\eta)+\varphi^{-1}D_\theta D_\eta.
\end{equation}
Factored as a quadratic in $D_1$, we have
\begin{equation}
0=D_1^2-(\varphi D_\theta+\varphi^{-1}D_\eta+1)D_1+D_\theta^2+D_\eta^2+\varphi^{-1}D_\theta D_\eta-D_\theta-D_\eta
\end{equation}
with solutions
\begin{equation}\label{xsol3}
D_1=\frac{1}{2}\left(\varphi D_\theta+\varphi^{-1}D_\eta+1\pm\sqrt{(\varphi D_\eta+\varphi^{-1}D_\eta+1)^2-4(D_\theta^2+D_\eta^2+\varphi^{-1}D_\theta D_\eta-D_\theta-D_\eta)}\right).
\end{equation}
For $D_1$ to be real, we require $D_\theta<6$ and $D_\eta<3$.  The solutions for $D_1$ in Equation (\ref{xsol3}) are maximized when $D_1<6$, hence $D$ is one of the numbers from Lemma \ref{lemroot5} and our claim follows from Lemma \ref{propfib}.


\subsubsection{$N=4$}\label{4sec}

\par The two cases to consider are if $\theta=-1$ and $\eta=\zeta_4$, and if $\theta=-\eta=\zeta_4$ without loss of generality.  The first case is completely described by Lemma \ref{core} and Lemma \ref{minusone} below, which includes arguments that overlap with the needs of Section \ref{8sec}.  The latter case is then treated separately, which contains more interesting sporadic examples.

\begin{lemma}\label{core}
Let $\mathcal{C}$ be a modular fusion category with $\mathrm{FPdim}(\mathcal{C})\in\mathbb{Z}$.  If the twists of $\mathcal{C}$ are $1$, $-1$, and $\zeta$, where $\zeta^8=1$ and $\zeta\neq\pm1$, then the unique minimal representative of the Witt class of $\mathcal{C}$ is a metric group of order $2$ or $4$.
\end{lemma}

\begin{proof}
We must have $D=2^m$ for some $m\in\mathbb{Z}_{\geq1}$, therefore $\mathcal{C}$ is nilpotent, hence weakly group-theoretical.  By \cite[Theorem 1.1]{MR3770935}, $\mathcal{C}$ is Witt equivalent to a product $\mathcal{B}\boxtimes\mathcal{D}$, where $\mathcal{D}$ is pointed and completely anisotropic, and $\mathcal{B}$ is an Ising modular fusion category or trivial.  We must have $\mathcal{B}\boxtimes\mathcal{D}$ is nontrivial because it is not possible $\xi=1$ with the given twists; we must have $\mathcal{B}$ trivial since all Ising modular fusion categories have twists which are $16$th roots of unity.  The classification of completely anisotropic metric groups \cite[Appendices A.3--A.5]{DGNO} completes our argument, with consideration to the restrictions on the twists of $\mathcal{C}$.
\end{proof}

\begin{lemma}\label{minusone}
Let $\mathcal{C}$ be a modular fusion category with $\mathrm{FPdim}(\mathcal{C})\in\mathbb{Z}$.  If the twists of $\mathcal{C}$ are $1$, $-1$, and $\zeta$, where $\zeta^8=1$ and $\zeta\neq\pm1$, then $\mathcal{C}$ is a metric group of order $4$.
\end{lemma}

\begin{proof}
Let $\mathrm{Rep}(G)\subset\mathcal{C}$ be a maximal Tannakian fusion subcategory which, if nontrivial, contains an invertible object $X$ of order $2$ by \cite[Proposition 8.2]{solvable}, for example.  The category of local modules over the regular algebra generated by $X$ is a modular fusion category satisfying the hypotheses of our claim, or is dimension $2$ since $\xi\not\in\mathbb{R}$.  Repeating this process a finite number of times produces one of the pointed modular fusion categories from Lemma \ref{core}.  Therefore it suffices to prove there do not exist modular fusion categories $\mathcal{D}$ satisfying the three conditions (1) $\dim(\mathcal{D})\in\{8,16\}$, (2) $\theta_X\in\{1,-1,\zeta\}$ for all simple $X\in\mathcal{D}$, and (3) $\mathcal{D}_A^0$ is braided equivalent to one of the pointed modular fusion categories in Lemma \ref{core}, where $A$ is the regular algebra of a Tannakian fusion subcategory of dimension $2$.
\par Proposition 1 of \cite{czenky2024frobeniusschurexponentbounds} implies we need to prove there does not exist a \emph{pointed} modular fusion category $\mathcal{D}$ satisfying (1), (2), and (3) from the preceding paragraph.  The case when $\dim(\mathcal{D})=8$ is immediately eliminated by perusing all possible modular data of rank $8$ \cite[Appendix D.7]{ng2023classificationmodulardatarank}.  When $\dim(\mathcal{D})=16$, the twists of simple objects are determined by \cite[Lemma 5.2]{MR3997136}: $\zeta$ with multiplicity $4$ and $\pm1$ with multiplicity $6$ each.  We first claim that if $X$ is an invertible object of order $2$, then $\theta_X=\pm1$.  Indeed, $X$ $\otimes$-generates a braided fusion subcategory of dimension $2$; if $\theta_X$ has order $4$, then $\mathcal{D}$ factors as a product violating the assumptions about the twists and there are no such dimension $2$ pointed categories with $\theta_X$ order $8$.  Now if there exists $X$ of order $4$ such that $\theta_X=1$, then $\theta_{X\otimes X}=\pm1$ and $\theta_{X^\ast}=1$.  There is no quadratic form on $C_4$ taking the values $1,1,1,-1$, hence $X$ $\otimes$-generates a Tannakian fusion subcategory of dimension $4$ and therefore $\mathcal{D}$ is a twisted double of a finite group, violating Lemma \ref{core}.  Moreover all simple objects of twist $1$ have order $2$.  Take any two such $X,Y$ of order $2$ and trivial twist.  If $\theta_{X\otimes Y}=-1$, then $X$ and $Y$ $\otimes$-generate a modular fusion subcategory of $\mathcal{D}$ and therefore $\mathcal{D}$ factors, in opposition to the assumed twists.  Alternatively, $\theta_{X\otimes Y}=1$ and $X$ and $Y$ $\otimes$-generate a Tannakian fusion subcategory of $\mathcal{D}$ and once again $\mathcal{D}$ is a twisted double of a finite group, violating Lemma \ref{core}.
\end{proof}

\par This completes the case $\theta=-1$ and $\eta$ has order four, since the hypothesis $\mathrm{FPdim}(\mathcal{C})\in\mathbb{Z}$ is automatic.  We will repeat the argument of Lemma \ref{minusone} in the case $\theta=-\eta=\zeta_4$ without loss of generality, but to cut down on the amount of brute-force searching through finite group modular data, we will prove the following ancillary results.  

\begin{lemma}\label{lastwan}
Let $G$ be a finite group and $\omega\in H^3(G,\mathbb{C}^\times)$ such that $\mathcal{Z}:=\mathcal{Z}(\mathrm{Vec}_G^\omega)$ has exactly three distinct twists $1$, $\theta$, and $\eta$.  Then either $G$ has exponent $3$ and the set of twists is $\{1,\zeta_3,\zeta_3^2\}$, or $G$ has exponent $2$ or $4$ and the set of twists is $\{1,\zeta_4,\zeta_4^3\}$.
\end{lemma}

\begin{proof}
Let $X\in\mathrm{Vec}_G^\omega$ be simple object of order $2$, if it exists.  Then $0=\mathrm{Tr}(\theta_{I(X)})=a+b\theta+c\eta$ for some $a,b,c\in\mathbb{Z}$.  If $b$ and $c$ were both nonzero, then $\theta^2=\eta^2$ \cite[Proposition 3.8]{MR4655273} and therefore $\theta=-\eta$ as $1,\theta,\eta$ must be distinct.  Moreover $\theta=-\eta=\zeta_4$ \cite[Corollary 3.10]{MR4655273}.  Otherwise, $c=0$ without loss of generality and $b\neq0$ since $X\not\cong\mathbbm{1}$ \cite[Note 3.9]{MR4655273}.  Thus $\theta=-1$ in this case and $\eta\neq\pm1$ is not real.  Let $Y\in\mathcal{Z}$ be any simple object with $\theta_Y=\eta$.  Then for any $X'\subset F(Y)$, we have $0=\mathrm{Tr}(\theta_{I(X')})=a-b+c\eta$ with $c\neq0$, which cannot occur since $\eta=(b-a)/c\in\mathbb{R}$.  We conclude that if there exists an element of order $2$ in $G$, then the twists of $\mathcal{Z}$ are $\{1,\zeta_4,\zeta_4^3\}$.  Moreover, $G$ is a finite group of exponent $2$ or $4$ in this case \cite[Theorem 9.2]{MR2313527}.  Alternatively, $G$ is a group of odd order.  Since $\xi$ is trivial, the twists are closed under square Galois conjugacy and include $1$.  The order of $G$ being odd implies all twists of $\mathcal{Z}$ have odd order which leaves only the possibilities $\{1,\zeta_3,\zeta_3^2\}$, $\{1,\zeta_5,\zeta_5^4\}$, or $\{1,\zeta_5^2,\zeta_5^3\}$.  The latter two cases are impossible, using the same argument as above: if $Y\in\mathcal{Z}$ is simple with $\theta_Y$ a primitive fifth root of unity, then $\mathrm{Tr}(I(X))$ is a $\mathbb{Z}_{\geq0}$-linear combination of only three distinct fifth roots of unity for any simple $X\subset F(Y)$, which cannot vanish.  This completes the argument.
\end{proof}

\begin{lemma}\label{anteriorpost}
Assume $G$ is a finite $2$-group and $\mathcal{Z}:=\mathcal{Z}(\mathrm{Vec}_G^\omega)$ has twists $1$, $\theta=\zeta_4$ and $\eta=\zeta_4^3$.  Then for all simple $X\in\mathcal{Z}$, $\theta_X=1$ if and only if $X$ lies in the canonical Lagrangian subcategory $\mathrm{Rep}(G)\subset\mathcal{Z}$, i.e.\ $X\subset I(\mathbbm{1})$. 
\end{lemma}

\begin{proof}
The first multiplicative central charge of $\mathcal{Z}$ is $1$, hence
\begin{equation}
|G|=\dim(\mathrm{Rep}(G))=\sum_{X\in\mathcal{O}(\mathcal{Z})}\dim(X)^2\theta_X=D_1+D_\theta\zeta_4+D_\eta\zeta_4^3.
\end{equation}
Therefore $D_\theta=D_\eta$ and the result is proven as $D_1=|G|$.
\end{proof}

\begin{lemma}\label{anterior}
Assume $G$ is a finite $2$-group and $\mathcal{Z}:=\mathcal{Z}(\mathrm{Vec}_G^\omega)$ has twists $1$, $\theta=\zeta_4$ and $\eta=\zeta_4^3$.  The simple summands of the canonical Lagrangian subcategory $\mathrm{Rep}(G)\subset\mathcal{Z}$ are self-dual. 
\end{lemma}

\begin{proof}
Let $Y\in\mathcal{O}(\mathrm{Rep}(G))$, and $Z\in\mathcal{O}(\mathcal{Z})\setminus\mathrm{Rep}(G)$ have $\theta_Z=\zeta_4$ without loss of generality.  We have $S_{Y,Z}=\zeta_4^3\sum_{W\subset Y\otimes Z}N_{Y,Z}^W\dim(W)\theta_W$ \cite[Proposition 8.13.8]{tcat}.  But $\mathrm{Rep}(G)$ is a fusion subcategory, thus $N_{Y,Z}^W=0$ unless $\theta_W\in\{\zeta_4,\zeta_4^3\}$ by Lemma \ref{anteriorpost}.  Therefore $S_{Y,Z}\in\mathbb{R}$ for all such simple $Z$, and $S_{Y,Z}=\dim(Y)\dim(Z)\in\mathbb{R}$ for all $Z\in\mathrm{Rep}(G)$ as $\mathrm{Rep}(G)$ is Tannakian.
\end{proof}

\begin{lemma}\label{c2}
Assume $G$ is a finite $2$-group, $G\not\cong C_2$, and $\mathcal{Z}:=\mathcal{Z}(\mathrm{Vec}_G^\omega)$ has twists $1$, $\theta=\zeta_4$ and $\eta=\zeta_4^3$.  If $X\in\mathcal{Z}$ is invertible,  then $X^2\cong\mathbbm{1}$ if and only if $\theta_X=1$.
\end{lemma}

\begin{proof}
If $\mathcal{Z}$ has an invertible object of nontrivial twist $\zeta\in\{\zeta_4,\zeta_4^3\}$ and order $2$, then $\mathcal{Z}$ factors as $\mathcal{Z}\simeq\mathcal{D}\boxtimes\mathcal{P}$ where $\mathcal{P}$ is pointed of dimension $2$ and $\mathcal{D}$ has twists $\{1,\zeta^{-1}\}$.  This would imply $G\cong C_2$ by Theorem \ref{th:twotwists}.  Conversely, if $\theta_X=1$, then $X\in\mathrm{Rep}(G)$ by Lemma \ref{anteriorpost} and the result follows from Lemma \ref{anterior}.
\end{proof}

\begin{lemma}\label{seetoo}
Assume $G$ is a finite $2$-group, $G\not\cong C_2$, and $\mathcal{Z}:=\mathcal{Z}(\mathrm{Vec}_G^\omega)$ has twists $1$, $\theta=\zeta_4$ and $\eta=\zeta_4^3$.  Any invertible object of $\mathcal{Z}$ of nontrivial twist has order $4$.
\end{lemma}

\begin{proof}
Let $X\in\mathcal{O}(\mathcal{Z})$ have nontrivial twist and denote by $\mathcal{D}$, the braided fusion subcategory of $\mathcal{Z}$ $\otimes$-generated by $X$.  If $\mathcal{D}$ were nondegenerately braided, then $G\cong C_2$ by the argument in Lemma \ref{c2}, against our assumptions.  Therefore $\mathcal{D}$ contains a nontrivial Tannakian subcategory $C_\mathcal{D}(\mathcal{D})$ which contains a nontrivial invertible object $\delta$ of order $2$ by Lemma \ref{c2}, which in turn implies $\mathcal{O}(C_\mathcal{D}(\mathcal{D}))=\{\mathbbm{1},\delta\}$.  Moreover, the de-equivariantization $\mathcal{D}_{C_2}$ by the $\otimes$-action of $\delta$ produces an anisotropic modular fusion category, which must be dimension $2$ by \cite[Proposition A.13]{DGNO}.  We are done as $\dim(\mathcal{D})=2\dim(\mathcal{D}_{C_2})$ \cite[Section 4.15]{tcat}.
\end{proof}

\begin{lemma}\label{propo2}
Assume $G$ is a finite $2$-group such that $\mathcal{Z}(\mathrm{Vec}_G^\omega)$ has exactly three distinct twists.  If there exists an invertible object $X\in\mathcal{Z}(\mathrm{Vec}_G^\omega)$ with nontrivial twist, then $\mathcal{Z}(\mathrm{Vec}_G^\omega)$ is pointed of dimenson $4$ or $16$.
\end{lemma}

\begin{proof}
We may assume $G\not\cong C_2$, so that $X$ $\otimes$-generates a braided fusion subcategory of $\mathcal{Z}(\mathrm{Vec}_G^\omega)$ of dimension $4$ by Lemma \ref{seetoo}, and $X^2:=X\otimes X$ has trivial twist by Lemma \ref{c2}.  Let $A:=\mathbbm{1}\oplus X^2$ be the corresponding regular $\mathrm{Rep}(C_2)$ algebra and consider the category of local $A$-modules $\mathcal{Z}(\mathrm{Vec}_G^\omega)_A^0$, which is a twisted double of a finite $2$-group $H$ \cite[Theorem 4.8]{drinfeld2007grouptheoretical} with simple objects having twists in the set $\{1,\zeta_3,\zeta_4\}$.  The $A$-module $A\otimes X$ is invertible of order $2$ with nontrivial twist, hence $H\cong C_2$ by Lemma \ref{seetoo}.
\end{proof}

\begin{example}\label{groupexamples}
Lemmas \ref{lastwan}--\ref{propo2} have shown we need only inspect the modular data of the twisted doubles of groups of order $2$ and $4$,  which gives exactly two examples of modular data satisfying the conditions of Lemma \ref{propo2}.  The smaller of the two, $\mathcal{Z}(\mathrm{Vec}_{C_2}^\omega)$ for nontrivial $\omega\in H^3(C_2,\mathbb{C}^\times)$ is simply the product of the two inequivalent pointed modular fusion categories of rank $2$.  The larger example is more interesting, having $C_4^2$ fusion rules, but the category does not factor.

\begin{align}
S&=\left[
\begin{array}{cccc}
1 & 1 & 1 & 1 \\
1 & 1 & -1 & -1 \\
1 & -1 & -1 & 1 \\
1 & -1 & 1 & -1
\end{array}
\right]&
T&=\mathrm{diagonal}(1, 1, \zeta_4^3, \zeta_4)
\end{align}

\begin{equation}
S=\left[\begin{array}{rrrrrrrrrrrrrrrr}
1 & 1 & 1 & 1 & 1 & 1 & 1 & 1 & 1 & 1 & 1 & 1 & 1 & 1 & 1 & 1 \\
1 & 1 & 1 & 1 & -1 & -1 & -1 & -1 & 1 & 1 & 1 & 1 & -1 & -1 & -1 & -1 \\
1 & 1 & 1 & 1 & 1 & 1 & 1 & 1 & -1 & -1 & -1 & -1 & -1 & -1 & -1 & -1 \\
1 & 1 & 1 & 1 & -1 & -1 & -1 & -1 & -1 & -1 & -1 & -1 & 1 & 1 & 1 & 1 \\
1 & -1 & 1 & -1 & -1 & 1 & -1 & 1 & \zeta_{4} & \zeta_{4} & \zeta_4^3 & \zeta_4^3 & \zeta_4^3 & \zeta_{4} & \zeta_4^3 & \zeta_{4} \\
1 & -1 & 1 & -1 & 1 & -1 & 1 & -1 & \zeta_4^3 & \zeta_4^3 & \zeta_{4} & \zeta_{4} & \zeta_4^3 & \zeta_{4} & \zeta_4^3 & \zeta_{4} \\
1 & -1 & 1 & -1 & -1 & 1 & -1 & 1 & \zeta_4^3 & \zeta_4^3 & \zeta_{4} & \zeta_{4} & \zeta_{4} & \zeta_4^3 & \zeta_{4} & \zeta_4^3 \\
1 & -1 & 1 & -1 & 1 & -1 & 1 & -1 & \zeta_{4} & \zeta_{4} & \zeta_4^3 & \zeta_4^3 & \zeta_{4} & \zeta_4^3 & \zeta_{4} & \zeta_4^3 \\
1 & 1 & -1 & -1 & \zeta_{4} & \zeta_4^3 & \zeta_4^3 & \zeta_{4} & -1 & 1 & 1 & -1 & \zeta_4^3 & \zeta_4^3 & \zeta_{4} & \zeta_{4} \\
1 & 1 & -1 & -1 & \zeta_{4} & \zeta_4^3 & \zeta_4^3 & \zeta_{4} & 1 & -1 & -1 & 1 & \zeta_{4} & \zeta_{4} & \zeta_4^3 & \zeta_4^3 \\
1 & 1 & -1 & -1 & \zeta_4^3 & \zeta_{4} & \zeta_{4} & \zeta_4^3 & 1 & -1 & -1 & 1 & \zeta_4^3 & \zeta_4^3 & \zeta_{4} & \zeta_{4} \\
1 & 1 & -1 & -1 & \zeta_4^3 & \zeta_{4} & \zeta_{4} & \zeta_4^3 & -1 & 1 & 1 & -1 & \zeta_{4} & \zeta_{4} & \zeta_4^3 & \zeta_4^3 \\
1 & -1 & -1 & 1 & \zeta_4^3 & \zeta_4^3 & \zeta_{4} & \zeta_{4} & \zeta_4^3 & \zeta_{4} & \zeta_4^3 & \zeta_{4} & -1 & 1 & 1 & -1 \\
1 & -1 & -1 & 1 & \zeta_{4} & \zeta_{4} & \zeta_4^3 & \zeta_4^3 & \zeta_4^3 & \zeta_{4} & \zeta_4^3 & \zeta_{4} & 1 & -1 & -1 & 1 \\
1 & -1 & -1 & 1 & \zeta_4^3 & \zeta_4^3 & \zeta_{4} & \zeta_{4} & \zeta_{4} & \zeta_4^3 & \zeta_{4} & \zeta_4^3 & 1 & -1 & -1 & 1 \\
1 & -1 & -1 & 1 & \zeta_{4} & \zeta_{4} & \zeta_4^3 & \zeta_4^3 & \zeta_{4} & \zeta_4^3 & \zeta_{4} & \zeta_4^3 & -1 & 1 & 1 & -1
\end{array}\right]
\end{equation}
\[T=\mathrm{diagonal}(1, 1, 1, 1, \zeta_4^3, \zeta_4, \zeta_4^3, \zeta_4, \zeta_4^3, \zeta_4, \zeta_4, \zeta_4^3, \zeta_4, \zeta_4^3, \zeta_4^3, \zeta_4)\]
\end{example}

\begin{lemma}\label{propo34}
Assume $G$ is a finite $2$-group such that $\mathcal{Z}(\mathrm{Vec}_G^\omega)$ has exactly three distinct twists.  Then $\mathcal{Z}(\mathrm{Vec}_G^\omega)$ is one of the two pointed modular fusion categories from Example \ref{groupexamples}, or the unique twisted double of $C_2^3$ with exactly three distinct twists \cite[Figure 4]{MR2333187}.
\end{lemma}

\begin{proof}
The argument in the proof of Lemma \ref{minusone} is immediately generalizable: it suffices to inspect the modular data of finite $2$-groups of increasing order until no examples are found.  The groups of orders $2$ and $4$ have already been inspected producing the modular data in Example \ref{groupexamples}.  The modular data of twisted doubles of groups of order $8$ was investigated in \cite{MR2333187} and there is a unique twisted double of $C_2^3$ with exactly three distinct twists \cite[Figure 4]{MR2333187}.  There are $14$ groups of order $16$ up to isomorphism but only $7$ have exponent $4$.  This leaves a paltry 204 sets of modular data to peruse either by hand or using the computer algebra software GAP, and none have exactly three distinct twists. 
\end{proof}

To finish the $N=4$ subsection, note that if $\theta=-\eta=\zeta_4$ without loss of generality, then $\xi=1$ or $\xi=\zeta_8$ up to Galois conjugacy of modular fusion categories by Figure \ref{fig:eighttwo}.  In the former case, $D_1=\sqrt{D}$ and $D\in\mathbb{Z}$, hence $\mathcal{C}$ is a twisted double of a finite group.  Lemma \ref{propo34} describes the three possible braided equivalence classes for $\mathcal{C}$ in this case.  Otherwise, as illustrated in Figure \ref{fig:eightonenowtoo}, $D_1=D_\theta+D_\eta$, hence $D=2D_1$.  But the Pythagorean theorem gives $2D_1^2=D$, hence $D_1=1$ and $D=2$, which cannot occur with three distinct twists.  We conclude there are only $3$ braided equivalence classes of modular fusion categories with three twists and $N=4$.
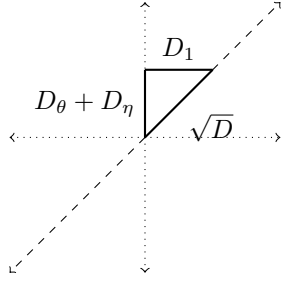
\begin{figure}[H]
\centering
\begin{equation*}
\begin{tikzpicture}[scale=0.18]
\draw[<->,dotted] (0,-10) -- (0,10);
\draw[<->,dotted] (-10,0) -- (10,0);
\draw[->,dashed] (0,0) -- (10,10);
\draw[->,dashed] (0,0) -- (-10,-10);
\draw[-,thick] (0,0) -- node[below right] {$\sqrt{D}$} (5,5);
\draw[thick] (0,5) -- node[above] {$D_1$} (5,5);
\draw[thick] (0,0) -- node[left] {$D_\theta+D_\eta$} (0,5);
\end{tikzpicture}
\end{equation*}
    \caption{A geometric view of $\tau_1=\zeta_8\sqrt{D}$ with twists $\{1,\zeta_4,\zeta_4^3\}$ and $\xi=\zeta_8$}%
    \label{fig:eightonenowtoo}%
\end{figure}


\subsubsection{$N=8$}\label{8sec}

\par In this case $n$ is divisible by $8$ \cite[Theorem 7.1]{MR2725181} and Lemma \ref{atemf} states that the twists of $\mathcal{C}$ are $1$, $-1$, and $\zeta_8$ up to Galois conjugacy of modular fusion categories and $\xi=\zeta_8$ or $\xi=\zeta_8^3$ by observing the possible $\xi$ in Figure \ref{fig:eighttwo}.  If $\xi=\zeta_8^3$, the argument for nonexistence is roughly the same as the final $N=4$ case above.  The calculation of $\xi=\zeta_8^3$ is illustrated in Figure \ref{fig:eightonenow}.
\begin{figure}[H]
\centering
\begin{equation*}
\begin{tikzpicture}[scale=0.18]
\draw[<->,dotted] (0,-10) -- (0,10);
\draw[<->,dotted] (-10,0) -- (10,0);
\draw[->,dashed] (0,0) -- (10,-10);
\draw[->,dashed] (0,0) -- (-10,10);
\draw[-,thick] (0,0) -- node[above right] {$\sqrt{D}$} (-5,5);
\draw[thick] (-10,0) -- node[above left] {$D_\eta$} (-5,5);
\draw[thick] (-10,0) -- node[below] {$D_1+D_\theta$} (0,0);
\end{tikzpicture}
\end{equation*}
    \caption{A geometric view of $\tau_1=\zeta_8^3\sqrt{D}$ with twists $\{1,-1,\zeta_8\}$ and $\xi=\zeta_8^3$}%
    \label{fig:eightonenow}%
\end{figure}
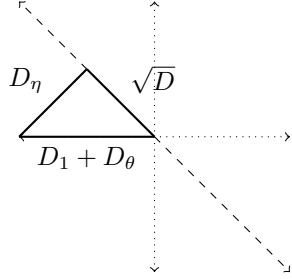
Geometrically, $D_\eta=\sqrt{D}=\sqrt{D_1+D_\theta+D_\eta}$, hence $D_\eta=(1/2)(1+\sqrt{1+4(D_1+D_\theta)})$.  Elementary trigonometry gives $D_\eta\sqrt{2}=D_1+D_\theta$, hence $D_\eta=(1/2)(1+\sqrt{1+4\sqrt{2}D_\eta})$ which implies $D_\eta=1+\sqrt{2}$ which cannot occur since $D_\eta$ has a Galois conjugate less than zero, and $D_\eta$ must be totally greater than or equal to $0$ \cite[Remark 2.5]{ENO}.

\par Lastly, if $\xi=\zeta_8$ and $D\in\mathbb{Z}$, then $D=4$ by Lemma \ref{minusone} and $D$ is pointed, so we may consider only $D\in\mathbb{Q}(\sqrt{2})\setminus\mathbb{Q}$.  In this case, $D_1=D_\theta$ and $D_\eta=\sqrt{D}=\sqrt{2D_1+D_\eta}$, hence $D_1=(1/2)D_\eta(D_\eta-1)$.  Note that $D_1$ is an algebraic $d$-number as $\mathcal{C}_1$ is closed under Galois conjugacy.  Moreover $D_1^2$ is an integer multiple of an algebraic unit, which implies $(D_\eta-1)^2$ is a rational multiple of a unit, i.e.\ $D_\eta-1$ is an algebraic $d$-number since $D_\eta-1$ is an algebraic integer.  Moreover, $D_\eta$ must be totally greater than or equal to $2$, or else $D_1=(1/2)D_\eta(D_\eta-1)$ has a Galois conjugate less than $(1/2)(2)(1)=1$, which cannot occur since the tensor unit exists.  Thus $D_\eta$ is a quadratic integer in $\mathbb{Q}(\sqrt{2})$ satisfying these conditions, which does not exist by Proposition \ref{dnumberz} below.


\appendix

\section{A categorical Galois action and algebraic $d$-numbers}\label{appa}

Let $\mathcal{C}$ be a modular fusion category.  The characters of the Grothendieck ring of $\mathcal{C}$ are in one-to-one correspondence with maps $Y\mapsto s_{Y,X}/s_{\mathbbm{1},X}$ for some $X\in\mathcal{O}(\mathcal{C})$.  This gives a well-defined permutation action on $\mathcal{O}(\mathcal{C})$ via the corresponding Galois conjugacy of characters.  Specifically, for each $\sigma\in\mathrm{Gal}(\overline{\mathbb{Q}}/\mathbb{Q})$ there is a unique permutation $\hat{\sigma}:\mathcal{O}(\mathcal{C})\to\mathcal{O}(\mathcal{C})$ given by $\sigma(s_{Y,X}/s_{\mathbbm{1},X})=s_{Y,\hat{\sigma}(X)}/s_{\mathbbm{1},\hat{\sigma}(X)}$ for all $X,Y\in\mathcal{O}(\mathcal{C})$.  When restricted to $X=\mathbbm{1}$ for example, we have the relation $\dim(\hat{\sigma}(X))^2=(D/\sigma(D))\sigma(\dim(X))^2$.  The interested reader can read more about this Galois action in \cite{gannoncoste,dong2015congruence}, for example.

\par The Galois action on the characters of modular fusion categories implies the modular data of a modular fusion category consists of cyclotomic numbers \cite[Theorem 8.14.7]{tcat} and certain numerical invariants of modular fusion categories are algebraic $d$-numbers, in the sense of \cite{codegrees}.  These are algebraic integers which generate Galois-invariant ideals in the ring of algebraic integers.  For identification and classification purposes, there are many alternative definitions \cite[Lemma 2.7]{codegrees}.  For the purposes of this manuscript, aside from the integers themselves, it suffices to consider algebraic $d$-numbers $\alpha$ such that $[\mathbb{Q}(\alpha):\mathbb{Q}]=2$.  In this case, $\alpha$ is an algebraic $d$-number if and only if $\alpha^2$ is an integer multiple of an algebraic unit \cite[Lemma 2.7(iv)]{codegrees} and algebraic units in real quadratic fields are multiplicatively generated by $\pm1$ and a chosen fundamental unit by Dirichlet's unit theorem.  Equivalently, with $\mathcal{N}(\alpha)$ and $\mathcal{T}(\alpha)$ the algebraic norm and trace of $\alpha$ in $\mathbb{Q}(\alpha)/\mathbb{Q}$, $\alpha$ is an algebraic $d$-number if and only if $\mathcal{T}(\alpha)^2$ is divisible by $\mathcal{N}(\alpha)$ \cite[Lemma 2.7(v)]{codegrees}.

\par Both the above characterizations of algebraic $d$-numbers will be used to prove the following ancillary results needed for specific lemmas in Section \ref{secsecthree}.

\begin{lemma}\label{lemroot5}
The algebraic $d$-numbers $\alpha\in\mathbb{Q}(\sqrt{5})$ whose norm is a power of $5$ and whose Galois conjugates lie between $1$ and $18$ are $(1/2)(5\pm\sqrt{5})$, $(5/2)(3\pm\sqrt{5})$, and $5$.
\end{lemma}

\begin{proof}
Denote $\varphi:=(1/2)(1+\sqrt{5})$.  The only possibilities for the norm of $\alpha$ in the extension $\mathbb{Q}(\sqrt{5})/\mathbb{Q}$ with the assumed bounds are $5$, $25$, and $125$.  In the first case, if $\alpha\neq5$, we must have $\alpha^2=5\varphi^m$ for some $m\in\mathbb{Z}$ \cite[Lemma 2.7]{codegrees}.  The bound $1\leq\alpha^2\leq20^2$ implies $-3\leq m\leq9$ and a finite check shows the only possibility is $\alpha=(1/2)(5\pm\sqrt{5})$.  Repeating this argument in the norm $25$ case gives $-6\leq m\leq5$ and thus $\alpha=(5/2)(3\pm\sqrt{5})$.  The final case gives $-10\leq m\leq2$ and none of these satisfy the hypotheses.
\end{proof}

\begin{lemma}\label{finlem}
Let $\alpha:=2^a\epsilon_2^b$ for some $a,b\in\mathbb{Z}_{\geq0}$ where $\epsilon_2:=1+\sqrt{2}$.  If $\alpha$ is totally greater than $2$ and $\alpha-1$ is an algebraic $d$-number, then $b=0$, i.e.\ $\alpha\in\mathbb{Z}$.
\end{lemma}

\begin{proof}
Since $\alpha-1$ is an algebraic $d$-number for all $\alpha\in\mathbb{Z}$, it suffices to prove the claim for $\alpha$ such that $\mathbb{Q}(\alpha)=\mathbb{Q}(\sqrt{2})$, hence $b>0$, and moreover $a>0$ since $\alpha$ is totally greater than $2$.  From this, it also follows that $b$ is even, and $2^a\geq2\epsilon_2^b=2^{\log_2(\epsilon_2)b+1}$.  Therefore, $b\leq\log_2(\epsilon_2)^{-1}(a-1)$ for all $a\in\mathbb{Z}_{\geq1}$.  The remainder of the proof is providing a lower bound on $b\in\mathbb{Z}_{\geq1}$.

\par Denote the trace and norm with respect to the field extension $\mathbb{Q}(\sqrt{2})/\mathbb{Q}$ by $\mathcal{T}$ and $\mathcal{N}$.  In this case, $\alpha$ has exactly one other Galois conjugate and we compute $\mathcal{T}(\alpha-1)=\mathcal{T}(\alpha)-2$ and $\mathcal{N}(\alpha-1)=\mathcal{N}(\alpha)-\mathcal{T}(\alpha)+1$.  From assumption, $\mathcal{N}(\alpha-1)$ is positive, hence $\mathcal{N}(\alpha)+1>\mathcal{T}(\alpha)$.  We have $\mathcal{N}(\alpha)=2^{2a}$ and $\mathcal{T}(\alpha)=2^a\mathcal{T}(\epsilon_2^b)=2^{a+1}m$ for some integer $m$ for brevity, since the trace of any algebraic integer in $\mathbb{Q}(\sqrt{2})$ is even, hence $m\leq2^{a-1}$. A lower bound on $m$ requires more subtle analysis.  To achieve this bound,
\begin{equation}
\mathcal{N}(\alpha-1)=\mathcal{N}(\alpha)-\mathcal{T}(\alpha)+1=2^{2a}-(2^{a+1}m-1).
\end{equation}
We have $\mathcal{T}(\alpha-1)=\mathcal{T}(\alpha)-2=2(2^am-1)$.  Therefore for $\alpha-1$ to be an algebraic $d$-number, we require
\begin{align}
\dfrac{(\mathcal{T}(\alpha-1))^2}{\mathcal{N}(\alpha-1)}=\dfrac{\mathcal{T}(\alpha)^2-4\mathcal{T}(\alpha)+4}{\mathcal{N}(\alpha)-\mathcal{T}(\alpha)+1}=4\dfrac{2^{2a}m^2-(2^{a+1}m-1)}{2^{2a}-(2^{a+1}m-1)}\in\mathbb{Z}.
\end{align}
Thus $(2^{2a}m^2-(2^{a+1}m-1))/(2^{2a}-(2^{a+1}m-1))\in\mathbb{Z}$ since the denominator is odd.  Simplifying,
\begin{equation}
\dfrac{2^{2a}m^2-(2^{a+1}m-1)}{2^{2a}-(2^{a+1}m-1)}=m^2+\dfrac{(2^{a+1}m-1)(m^2-1)}{2^{2a}-(2^{a+1}m-1)},\label{finaleq}
\end{equation}
and so $\alpha-1$ is an algebraic $d$-number only if $(2^{a+1}m-1)(m^2-1)/(2^{2a}-(2^{a+1}m-1))\in\mathbb{Z}$.  We have $\gcd(2^{2a}-(2^{a+1}m-1),(2^{a+1}m-1))=\gcd(2^{2a},2^{a+1}m-1)=1$ since $2^{a+1}m-1$ is odd.  So if the right-hand side of Equation (\ref{finaleq}) is an integer, we must have
\begin{align}
2^{2a}-(2^{a+1}m-1)\leq m^2-1\leq 2^{2(a-1)}-1
\end{align}
where the second inequality follows from $m\leq2^{a-1}$.  Thus $2^{a+1}m\geq2^{2a}-2^{2(a-1)}+2$, and therefore
\begin{align}
m\geq\dfrac{2^{2a}-2^{2(a-1)}+2}{2^{a+1}}=3\cdot2^{a-3}+2^{-a}>3\cdot2^{a-3}.
\end{align}
As a result, $\mathcal{T}(\epsilon_2^b)=2m>3\cdot2^{a-2}$ and question remains if this is possible with the upper bound on $b$ from the first argument of the proof.  It is clear $\mathcal{T}(\epsilon_2^b)<\epsilon_2^b=2^{\log_2(\epsilon_2)b}$ since $1-\sqrt{2}<0$, hence in sum, we require $3\cdot2^{a-2}<\mathcal{T}(\epsilon_2^b)<2^{\log_2(\epsilon_2)b}\leq2^{a-1}$.  But this implies $3/2<1$ so no such $a\in\mathbb{Z}_{\geq1}$ exists.
\end{proof}

\begin{lemma}\label{bells}
Let $\alpha:=2^a\epsilon_2^b\sqrt{2}$ for some $a,b\in\mathbb{Z}_{\geq0}$ where $\epsilon_2:=1+\sqrt{2}$.  If $\alpha$ is totally greater than $2$ and $\alpha-1$ is an algebraic $d$-number, then $b=0$, i.e.\ $\alpha\in\mathbb{Z}$.
\end{lemma}

\begin{proof}
The proof follows that of Lemma \ref{finlem}.  From $\alpha$ being totally greater than or equal to $2$, it follows that $b$ is odd, and $2^a\geq2^{1/2}\epsilon_2^b=2^{\log_2(\epsilon_2)b+1/2}$.  Therefore, $b\leq\log_2(\epsilon_2)^{-1}(a-1/2)$ for all $a\in\mathbb{Z}_{\geq1}$.  The remainder of the proof is providing a lower bound on $b\in\mathbb{Z}_{\geq1}$.

\par Denote the trace and norm with respect to the field extension $\mathbb{Q}(\sqrt{2})/\mathbb{Q}$ by $\mathcal{T}$ and $\mathcal{N}$.  In this case, $\alpha$ has exactly one other Galois conjugate and we compute $\mathcal{T}(\alpha-1)=\mathcal{T}(\alpha)-2$ and $\mathcal{N}(\alpha-1)=\mathcal{N}(\alpha)-\mathcal{T}(\alpha)+1$.  From assumption, $\mathcal{N}(\alpha-1)$ is positive, hence $\mathcal{N}(\alpha)+1>\mathcal{T}(\alpha)$.  We have $\mathcal{N}(\alpha)=2^{2a+1}$ and $\mathcal{T}(\alpha)=2^a\mathcal{T}(\epsilon_2^b\sqrt{2})=2^{a+2}m$ for some integer $m$ for brevity, since $\mathcal{T}(\epsilon_2^b\sqrt{2})=4\cdot U_b$ where $U_0:=0$, $U_1:=1$, and $U_n:=2\cdot U_{n-1}+U_{n-2}$.  , hence $m\leq2^{a-1}$. As before,
\begin{equation}
\mathcal{N}(\alpha-1)=\mathcal{N}(\alpha)-\mathcal{T}(\alpha)+1=2^{2a+1}-(2^{a+2}m-1).
\end{equation}
We have $\mathcal{T}(\alpha-1)=\mathcal{T}(\alpha)-2=2(2^{a+1}m-1)$.  Therefore for $\alpha-1$ to be an algebraic $d$-number, we require
\begin{align}
\dfrac{(\mathcal{T}(\alpha-1))^2}{\mathcal{N}(\alpha-1)}=\dfrac{\mathcal{T}(\alpha)^2-4\mathcal{T}(\alpha)+4}{\mathcal{N}(\alpha)-\mathcal{T}(\alpha)+1}=4\dfrac{2^{2(a+1)}m^2-(2^{a+2}m-1)}{2^{2a+1}-(2^{a+2}m-1)}\in\mathbb{Z}.
\end{align}
The simplification is only slightly different in this case:
\begin{equation}
4\dfrac{2^{2(a+1)}m^2-(2^{a+2}m-1)}{2^{2a+1}-(2^{a+2}m-1)}=m^2+\dfrac{(2^{a+2}m-1)(2m^2-1)}{2^{2a+1}-(2^{a+2}m-1)},\label{finaleqb}
\end{equation}
and so, as above, we must have
\begin{align}
2^{2a+1}-(2^{a+2}m-1)\leq 2m^2-1\leq 2\cdot2^{2(a-1)}-1
\end{align}
where the second inequality follows from $m\leq2^{a-1}$.  Thus $2^{a+2}m\geq2^{2a+1}-2^{2a-1}+2$, and therefore
\begin{align}
m\geq\dfrac{2^{2a+1}-2^{2a-1}+2}{2^{a+2}}=3\cdot2^{a-3}+2^{-a-1}.
\end{align}
As a result, $\mathcal{T}(\epsilon_2^b\sqrt{2})=4m>3\cdot2^{a-1}+2^{-(a-1)}$ and question remains if this is possible with the upper bound on $b$ from the first argument of the proof.  It is clear $\mathcal{T}(\epsilon_2^b\sqrt{2})<\epsilon_2^b\sqrt{2}+2^{-(a-1)}=2^{\log_2(\epsilon_2)b+1/2}+2^{-(a-1)}$ since $(1-\sqrt{2})^a<2^{-(a-1)}$, hence in sum, we require \begin{equation}
3\cdot2^{a-1}+2^{-(a-1)}<\mathcal{T}(\epsilon_2^b)<2^{\log_2(\epsilon_2)b+1/2}+2^{-(a-1)}\leq2^a+2^{-(a-1)}.
\end{equation}
But this implies $3/2<1$ so no such $a\in\mathbb{Z}_{\geq1}$ exists.
\end{proof}

\begin{proposition}\label{dnumberz}
Let $\alpha$ be an algebraic $d$-number in $\mathbb{Q}(\sqrt{2})$ whose norm is a power of $2$.  If $\alpha$ is totally greater than or equal to $2$ and $\alpha-1$ is an algebraic $d$-number, then $\alpha\in\mathbb{Z}$.
\end{proposition}

\begin{proof}
Any algebraic $d$-number $\alpha$ in $\mathbb{Q}(\sqrt{2})$ whose norm is a power of $2$ has $\alpha^2=2^a\epsilon_2^m$ for some $a\in\mathbb{Z}_{\geq0}$ and $m\in\mathbb{Z}$.  If the norm of $\alpha$ is a perfect square, then $\alpha$ has the form as in Lemma \ref{finlem}, otherwise it has the form as in \ref{bells}.
\end{proof}

\bibliographystyle{plain}
\bibliography{bib}

\begin{thebibliography}{10}

\bibitem{MR3632091}
Paul Bruillard, Siu-Hung Ng, Eric~C. Rowell, and Zhenghan Wang.
\newblock On classification of modular categories by rank.
\newblock {\em Int. Math. Res. Not. IMRN}, (24):7546--7588, 2016.

\bibitem{MR3486174}
Paul Bruillard, Siu-Hung Ng, Eric~C. Rowell, and Zhenghan Wang.
\newblock Rank-finiteness for modular categories.
\newblock {\em J. Amer. Math. Soc.}, 29(3):857--881, 2016.

\bibitem{gannoncoste}
Antoine Coste and Terry Gannon.
\newblock Remarks on {G}alois symmetry in rational conformal field theories.
\newblock {\em Physics Letters B}, 323(3):316--321, 1994.

\bibitem{MR1770077}
Antoine Coste, Terry Gannon, and Philippe Ruelle.
\newblock Finite group modular data.
\newblock {\em Nuclear Phys. B}, 581(3):679--717, 2000.

\bibitem{czenky2024frobeniusschurexponentbounds}
Agustina Czenky, Julia Plavnik, and Andrew Schopieray.
\newblock On {F}robenius-{S}chur exponent bounds, 2024.

\bibitem{davidovich2013arithmetic}
Orit Davidovich, Tobias Hagge, and Zhenghan Wang.
\newblock On arithmetic modular categories.
\newblock {\em arXiv e-prints}, page arXiv:1305.2229, May 2013.

\bibitem{MR3039775}
Alexei Davydov, Michael M\"{u}ger, Dmitri Nikshych, and Victor Ostrik.
\newblock The {W}itt group of non-degenerate braided fusion categories.
\newblock {\em J. Reine Angew. Math.}, 677:135--177, 2013.

\bibitem{dong2015congruence}
Chongying Dong, Xingjun Lin, and Siu-Hung Ng.
\newblock Congruence property in conformal field theory.
\newblock {\em Algebra Number Theory}, 9(9):2121--2166, 2015.

\bibitem{drinfeld2007grouptheoretical}
Vladimir Drinfeld, Shlomo Gelaki, Dmitri Nikshych, and Victor Ostrik.
\newblock Group-theoretical properties of nilpotent modular categories, 2007.

\bibitem{DGNO}
Vladimir Drinfeld, Shlomo Gelaki, Dmitri Nikshych, and Victor Ostrik.
\newblock On braided fusion categories. {I}.
\newblock {\em Selecta Math. (N.S.)}, 16(1):1--119, 2010.

\bibitem{MR1227715}
W.~Eholzer.
\newblock Fusion algebras induced by representations of the modular group.
\newblock {\em Internat. J. Modern Phys. A}, 8(20):3495--3507, 1993.

\bibitem{MR1354262}
Wolfgang Eholzer.
\newblock On the classification of modular fusion algebras.
\newblock {\em Comm. Math. Phys.}, 172(3):623--659, 1995.

\bibitem{tcat}
Pavel Etingof, Shlomo Gelaki., Dmitri Nikshych, and Victor Ostrik.
\newblock {\em Tensor Categories}.
\newblock Mathematical Surveys and Monographs. American Mathematical Society,
  2015.

\bibitem{solvable}
Pavel Etingof, Dmitri Nikshych, and Victor Ostrik.
\newblock Weakly group-theoretical and solvable fusion categories.
\newblock {\em Advances in Mathematics}, 226(1):176--205, 2011.

\bibitem{ENO}
Pavel Etingof, Dmitri Nikshych, and Viktor Ostrik.
\newblock On fusion categories.
\newblock {\em Ann. of Math. (2)}, 162(2):581--642, 2005.

\bibitem{MR4836055}
Terry Gannon and Andrew Schopieray.
\newblock Algebraic number fields generated by dimensions in fusion rings.
\newblock {\em Commun. Number Theory Phys.}, 18(4):705--743, 2024.

\bibitem{MR2383894}
Shlomo Gelaki and Dmitri Nikshych.
\newblock Nilpotent fusion categories.
\newblock {\em Adv. Math.}, 217(3):1053--1071, 2008.

\bibitem{MR2333187}
Christopher Goff, Geoffrey Mason, and Siu-Hung Ng.
\newblock On the gauge equivalence of twisted quantum doubles of elementary
  abelian and extra-special 2-groups.
\newblock {\em J. Algebra}, 312(2):849--875, 2007.

\bibitem{MR4257620}
Angus Gruen and Scott Morrison.
\newblock Computing modular data for pointed fusion categories.
\newblock {\em Indiana Univ. Math. J.}, 70(2):561--593, 2021.

\bibitem{MR1936496}
Alexander Kirillov, Jr. and Viktor Ostrik.
\newblock On a {$q$}-analogue of the {M}c{K}ay correspondence and the {ADE}
  classification of {$\mathfrak{sl}_2$} conformal field theories.
\newblock {\em Adv. Math.}, 171(2):183--227, 2002.

\bibitem{data}
Micha\"el Mignard and Peter Schauenburg.
\newblock Modular categories are not determined by their modular data.
\newblock {\em Lett. Math. Phys.}, 111(60), 2021.

\bibitem{MR3770935}
Sonia Natale.
\newblock The core of a weakly group-theoretical braided fusion category.
\newblock {\em Internat. J. Math.}, 29(2):1850012, 23, 2018.

\bibitem{MR4630478}
Siu-Hung Ng, Eric~C. Rowell, Zhenghan Wang, and Xiao-Gang Wen.
\newblock Reconstruction of modular data from {${\rm SL}_2(\Bbb Z)$}
  representations.
\newblock {\em Comm. Math. Phys.}, 402(3):2465--2545, 2023.

\bibitem{ng2023classificationmodulardatarank}
Siu-Hung Ng, Eric~C. Rowell, and Xiao-Gang Wen.
\newblock Classification of modular data up to rank 11, 2023.

\bibitem{MR2313527}
Siu-Hung Ng and Peter Schauenburg.
\newblock Frobenius-{S}chur indicators and exponents of spherical categories.
\newblock {\em Adv. Math.}, 211(1):34--71, 2007.

\bibitem{MR2725181}
Siu-Hung Ng and Peter Schauenburg.
\newblock Congruence subgroups and generalized {F}robenius-{S}chur indicators.
\newblock {\em Comm. Math. Phys.}, 300(1):1--46, 2010.

\bibitem{MR3997136}
Siu-Hung Ng, Andrew Schopieray, and Yilong Wang.
\newblock Higher {G}auss sums of modular categories.
\newblock {\em Selecta Math. (N.S.)}, 25(4):Paper No. 53, 32, 2019.

\bibitem{MR4389082}
Siu-Hung Ng, Yilong Wang, and Qing Zhang.
\newblock Modular categories with transitive {G}alois actions.
\newblock {\em Comm. Math. Phys.}, 390(3):1271--1310, 2022.

\bibitem{MR444787}
Alexandre Nobs.
\newblock Die irreduziblen {D}arstellungen der {G}ruppen {$SL\sb{2}(Z\sb{p})$},
  insbesondere {$SL\sb{2}(Z\sb{2})$}. {I}.
\newblock {\em Comment. Math. Helv.}, 51(4):465--489, 1976.

\bibitem{MR444788}
Alexandre Nobs and J\"{u}rgen Wolfart.
\newblock Die irreduziblen {D}arstellungen der {G}ruppen {$SL\sb{2}(Z\sb{p})$},
  insbesondere {$SL\sb{2}(Z\sb{p})$}. {II}.
\newblock {\em Comment. Math. Helv.}, 51(4):491--526, 1976.

\bibitem{codegrees}
Victor Ostrik.
\newblock On formal codegrees of fusion categories.
\newblock {\em Math. Res. Lett.}, 16(5):895--901, 2009.

\bibitem{MR3427429}
Victor Ostrik.
\newblock Pivotal fusion categories of rank 3.
\newblock {\em Mosc. Math. J.}, 15(2):373--396, 405, 2015.

\bibitem{ostrikremarks}
Victor Ostrik.
\newblock Remarks on global dimensions of fusion categories.
\newblock In {\em Tensor categories and {H}opf algebras}, volume 728 of {\em
  Contemp. Math.}, pages 169--180. Amer. Math. Soc., Providence, RI, 2019.

\bibitem{ostrik}
Viktor Ostrik.
\newblock Fusion categories of rank 2.
\newblock {\em Math. Res. Lett.}, 10(2-3):177--183, 2003.

\bibitem{MR4834521}
Julia Plavnik, Andrew Schopieray, Zhiqiang Yu, and Qing Zhang.
\newblock Modular {T}ensor {C}ategories, {S}ubcategories, and {G}alois
  {O}rbits.
\newblock {\em Transform. Groups}, 29(4):1623--1648, 2024.

\bibitem{MR4079742}
Andrew Schopieray.
\newblock Lie theory for fusion categories: a research primer.
\newblock In {\em Topological phases of matter and quantum computation}, volume
  747 of {\em Contemp. Math.}, pages 1--26. Amer. Math. Soc., Providence, RI,
  2020.

\bibitem{MR4655273}
Andrew Schopieray.
\newblock Categorification of integral group rings extended by one dimension.
\newblock {\em J. Lond. Math. Soc. (2)}, 108(4):1617--1641, 2023.

\bibitem{MR4202289}
Zheyan Wan and Yilong Wang.
\newblock Classification of spherical fusion categories of {F}robenius-{S}chur
  exponent 2.
\newblock {\em Algebra Colloq.}, 28(1):39--50, 2021.

\bibitem{MR4564492}
Zhiqiang Yu.
\newblock Pre-modular fusion categories of global dimension {$p^2$}.
\newblock {\em J. Algebra}, 624:63--92, 2023.

\bibitem{MR4793461}
Zhiqiang Yu.
\newblock On the realization of a class of {${\rm SL}(2,\Bbb{Z})$}
  representations.
\newblock {\em J. Noncommut. Geom.}, 18(4):1521--1542, 2024.

\end{thebibliography}

\end{document}